\documentclass[leqno]{siamltex704}
\usepackage{amssymb,amsmath,graphicx,amscd,mathrsfs}
\usepackage{color,xcolor,amsmath}
\usepackage{amsmath}
\usepackage{graphicx}
\usepackage{mathrsfs}
\usepackage{float}
\usepackage{amsfonts,amssymb}
\usepackage{dsfont}
\usepackage{pifont}
\usepackage{hyperref}
\usepackage{multirow}
\usepackage{graphicx} 
\usepackage{epstopdf}
\usepackage{bm}

\numberwithin{equation}{section}
\def\3bar{{|\hspace{-.02in}|\hspace{-.02in}|}}
\def\E{{\mathcal{E}}}
\def\T{{\mathcal{T}}}

\def\dQ{{\mathbb{Q}}}
\def\bQ{{\mathbf{Q}}}

\def\btau{\boldsymbol{\tau}}

\def\b0{\boldsymbol{0}}
\def\sumT{\sum_{T\in\mathcal{T}_h}}     
\def\sumTa{\sum_{T_s\in\mathcal{T}_{s,h}}}     
\def\trb{|\!|\!|}

\def\bw{{\mathbf{w}}}
\def\bu{{\mathbf{u}}}
\def\bv{{\mathbf{v}}}
\def\bn{{\mathbf{n}}}
\def\be{{\mathbf{e}}}
\def\bf{{\mathbf{f}}}

\newtheorem{remark}{Remark}[section]
\newtheorem{algorithm1}{WG-MFEM Scheme}

 \newcommand{\Real}{\mathbb{R}}

 \newcommand{\Tensor}{\mathbb{T}}

{\color{blue}
\title{ A weak Galerkin-mixed finite element method for the Stokes-Darcy problem}
}

 \author{
 Hui Peng\thanks{School of Mathematics, Jilin University, Changchun,
China.  penghui17@mails.jlu.edu.cn.}
\and
 Qilong Zhai\thanks{(corresponding author)~School of Mathematics, Jilin University, Changchun,
    China.  zhaiql@jlu.edu.cn. The research of this author is
     supported in part by China National Natural Science Foundation(11901015). }
 \and Ran Zhang\thanks{School of Mathematics, Jilin University, Changchun,
    China.  zhangran@mail.jlu.edu.cn. The research of this author is
     supported in part by China National Natural Science Foundation (11971198, 91630201, 11871245, 11771179, 11826101), and by the Program for Cheung Kong Scholars(Q2016067), Key Laboratory of Symbolic Computation and Knowledge Engineering of Ministry of Education, Jilin University, Changchun, 130012, P.R.China.
     } \and
  Shangyou Zhang\thanks{Department of Mathematical Sciences,
     University of Delaware, Newark, DE 19716, U.S.A. szhang@udel.edu. }  }

\begin{document}

\maketitle

\begin{abstract}
In this paper, we propose a new numerical scheme for the coupled Stokes-Darcy model with Beavers-Joseph-Saffman interface condition. We use the weak Galerkin method to discretize the Stokes equation and the mixed finite element method to the Darcy equation. A discrete inf-sup condition is proved and optimal error estimates are also derived. Numerical experiments validate the theoretical analysis.
\end{abstract}

\begin{keywords} weak Galerkin finite element methods, mixed finite element methods,
 weak gradient,  coupled Stokes-Darcy problems
\end{keywords}

\begin{AMS}
Primary, 65N30, 65N15, 65N12; Secondary, 35B45, 35J50
\end{AMS}

\section{Introduction}
The coupling of fluid flow and porous media flow has received an increasing attention during the last decade. This coupled flow arises in many fields, such as the transport of contaminants through steams in environment, the filtration of flood through vessel walls in physiology, and some technologies involving fluid filter in industrial. Interested readers may refer to \cite{application3-first, application2, application4,application1} and the reference therein.

The mathematical model of such a coupled problem consists of Stokes equations in the fluid region and the Darcy's law in the porous medium.  Appropriate interface conditions, namely mass conservation, balance of force and the Beavers-Joseph-Saffman condition \cite{interface1, interface3, interface2} are imposed on the interface between the free flow region and porous medium flow region.

Early studies on numerical simulations and error analysis for the coupled Stokes-Darcy problem can be found in $\cite{early work1, early work2}$. In a comprehensive study presented in \cite{application3-first}, Discacciati et al. analyze a standard velocity-pressure formulation in the Stokes region and a second order primal elliptic problem in the Darcy region. Continuous finite element methods are used in both space. In $\cite{well-posedness}$, Layton et al. consider a mixed formulation in Darcy region, which involves the velocity and pressure simultaneously. They prove the existence and uniqueness of a weak solution to the mixed Stokes-Darcy system. Continuous finite element method employed in Stokes region and the mixed finite element method used in Darcy region. Later, the discontinuous Galerkin(DG) methods are applied to this problem \cite{half-DG,DG}. The work combines DG method for the Stokes equations with the mixed finite element method for the Darcy equation is proposed in $\cite{half-DG}$. Analysis of the DG method for both Stokes and Darcy equations introduced in $\cite{DG}$. In addition, preconditioning techniques  are also used for the coupled flow $\cite{pre1}$. More recent studies concerning the Stokes-Darcy problem can be found in $\cite{multi-grid1, boundary integral1, domain decomposition1, WG1, WG3, WG2, multi-grid2, time-stepping, boundary integral2, domain decomposition2, multi-grid3}$.

The weak Galerkin (WG) finite element method is proposed in $\cite{WG proposed}$ by Wang and Ye for the second order elliptic equation. {\color{black}They introduce} totally discontinuous weak functions and corresponding weak differential operators. Numerical implementation of WG methods for different models with more general finite element partitions is discussed in $\cite{WG-implement}$. The WG scheme is designed on arbitrary shape of polygons in 2D or polyhedra in 3D with certain shape regularity by introducing a stabilizer in $\cite{Wang2014a}$.
 Unified study for WG methods and other discontinuous Galerkin methods is presented in $\cite{uni1,uni2}$. In the past few years, the WG method is widely applied to many partial differential problems because of its flexibility and efficiency. The corresponding work can be found in $\cite{wg-application5, wg-application1, wg-application2, wg-application3, wg-application4}$.

Recently, WG methods are developed for solving the Stokes-Darcy model. In $\cite{WG2}$, the coupled system is described by Stokes equations in primal velocity-pressure formulation and the Darcy's law in primal pressure formulation. The piecewise constant elements are used to approximate the velocity, hydraulic and pressure. Furthermore, the same formulation is discussed in $\cite{WG3}$, different choices of WG finite element spaces are investigated, the classical meshes in $\cite{WG2}$ are extended to general polygonal meshes. In $\cite{WG1}$, the authors consider the mixed formulation in the Darcy region, both the Stokes region and Darcy region involve the velocity and the pressure. Strong coupling of the Stokes-Darcy system is achieved in the discrete space by using the WG approach.

As mentioned above, we can see that WG methods show a high flexibility for dealing with the Stokes-Darcy problem. However, the decoupling of the elements leads to an increase in the total degrees of freedom, which limits the practical utility of WG methods, especially in high order approximations. The aim of this article is to introduce a new numerical scheme with fewer number of degrees of freedom for the same mixed Stokes-Darcy formulation as $\cite{WG1}$. To this end, we use different finite element discretizations for the two regions. The WG method is still employed to approximate the velocity and the pressure in Stokes region. A summary for the features of WG methods to solve Stokes equation is provided in $\cite{Wang2015b}$. As for the Darcy region, the same unknowns are approximated by the mixed finite element (MFEM) method, which is different from the WG approximation in $\cite{WG1}$. Readers may refer to, e.g. $\cite{compare}$ for a comparison of degrees of freedom between WG methods and MFEM methods. Several standard mixed finite element spaces can be chosen, such as RT spaces $\cite{RT}$, BDM spaces $\cite{BDM}$, BDFM spaces $\cite{BDFM}$ and so on. The efficiency of the MFEM has been demonstrated in $\cite{MFE1, MFE2, MFE3}$.  Lagrange multiplier is introduced to impose the continuity of the velocity. The benefit of our approach is the possibility of combining the efficiency of the MFEM methods for Darcy problem with the flexibility of WG methods for Stokes problem. However, the combination of these two different finite element methods makes the proof process more complex for the inf-sup condition than $\cite{WG1}$. Inspired by the work in $\cite{half-DG}$, we construct two local projection operators in different region to prove it.

The rest of the paper is organized as follows. In the next section, we present the model problem, some notations and function spaces. In Section 3, we introduce weak Galerkin methods and construct WG-MFEM numerical scheme for the Stokes-Darcy problem. The well-posedness of the scheme is analyzed in Section 4. We derive the error estimates for the corresponding numerical approximations in Section 5. Finally, some numerical examples are presented to show the good performance of the developed algorithm in Section 6.

\section{Model Problem and Weak Formulation}
Let $\Omega$  be a bounded domain in $\mathbb{R}^2,$ subdivided into a free fluid region $\Omega_s$ and a porous region $\Omega_d$. Denote by $\Gamma=\partial \Omega_s\cap\partial \Omega_d$ the interface, and by $\Gamma_s=\partial \Omega_s\setminus \Gamma$, $\Gamma_d=\partial \Omega_d\setminus \Gamma$ the outer boundary. Moreover, let $\bn$ and $\bm{\tau}$ be the unit normal and tangential vectors to $\Gamma$, respectively, see Fig.\ref{domain-illustration}.

\begin{figure}[htb]\begin{center}\setlength\unitlength{1in}
    \begin{picture}(3,2.4)
 \put(0,0){\includegraphics[width=3.3in]{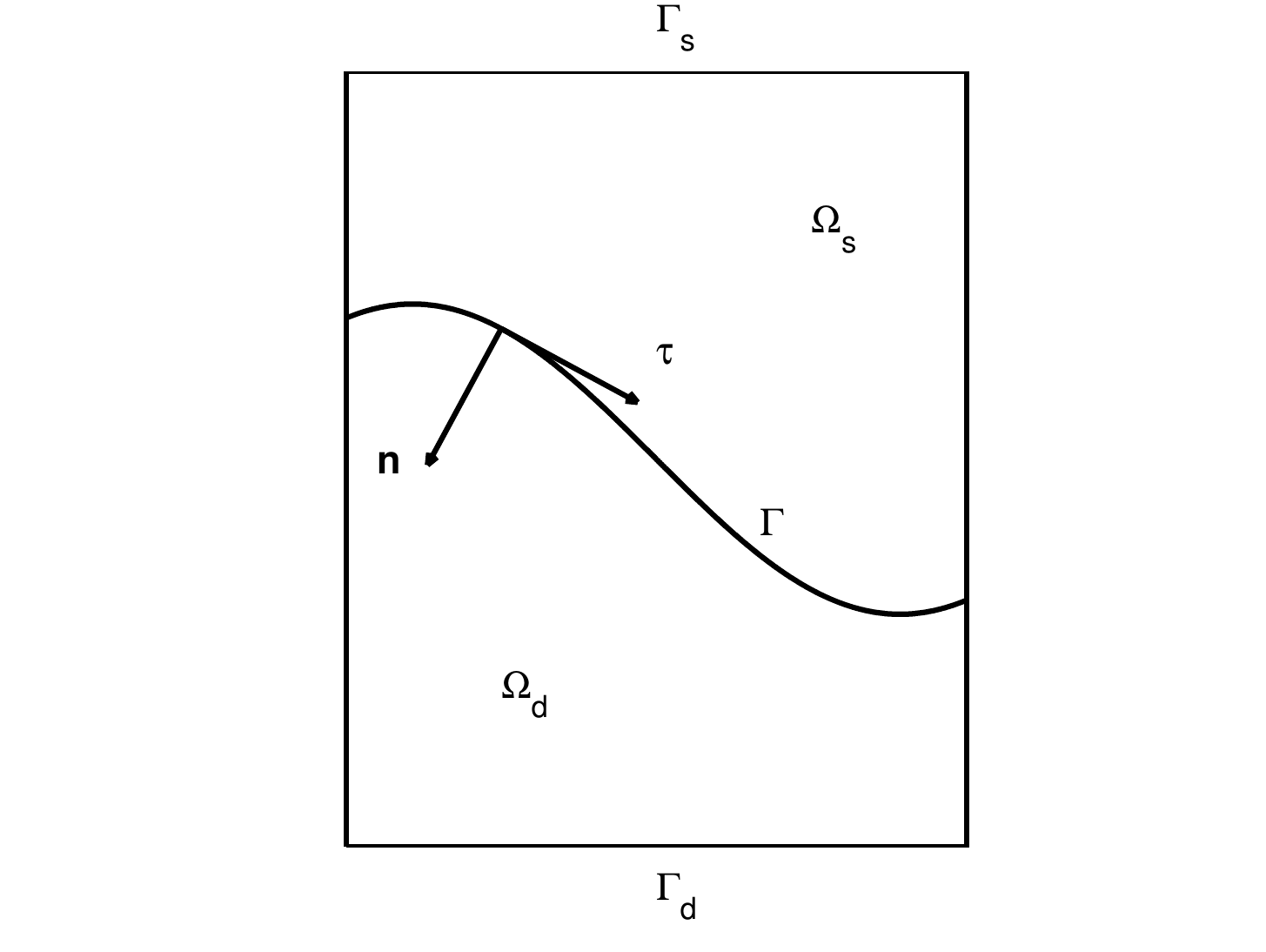}}
    \end{picture}
\caption{ {\color{black}Domain schematic for Stokes-Darcy coupled flow.}  } \label{domain-illustration}
\end{center}
\end{figure}

In $\Omega_s$, the fluid flow is governed by Stokes equations.
\begin{eqnarray}
\label{problem-eq1}
-\nabla\cdot\Tensor(\bu_s,p_s)&=&\mathbf{f}_s \quad \text{in }\Omega_s,
\\\label{problem-eq2}
\nabla\cdot \bu_s&=& 0 \quad \text{in }\Omega_s,
\\\label{problem-eq5}
\bu_s &=& 0 \quad \text{on }\Gamma_s,
\end{eqnarray}
where $\mathbb{T}$ is the stress tensor, $\mathbb{T}(\bu_s,p_s)=2\nu D(\bu_s)-p_s I$ and $ {D}(\bu_s)=\frac{1}{2}(\nabla \bu_s+\nabla^T\bu_s)$, $\nu$ is the kinematic viscosity of the fluid and $I$ is the identity matrix. $\mathbf{f}_s$ is a given external body force.

In $\Omega_d$, the porous media flow is governed by Darcy's law.
\begin{eqnarray}
\label{problem-eq3}
\nabla \cdot \bu_d&=&f_d ~~~~~~~~~~ \text{in}~\Omega_d,
\\\label{problem-eq4}
\bu_d&=&-\mathbb{K}\nabla p_d \quad \text{in} ~\Omega_d,
\\\label{problem-eq6}
\bu_d\cdot\bn_d&=&0~~~~~~~~~~~ \text{on }\Gamma_d,
\end{eqnarray}
where $\mathbb{K}$ is the symmetric positive-defined permeability tensor, $f_d$ is the source term and satisfies the following condition
\begin{eqnarray*}
\int_{\Omega_d}f_d=0.
\end{eqnarray*}

The interface conditions on $\Gamma$ consist of three parts.
\begin{eqnarray}\label{problem-eq7}
\bu_s\cdot\bn&=&\bu_d\cdot\bn ~~~~~~~~~~~ \text{on }\Gamma,
\\\label{problem-eq8}
-\Tensor(\bu_s,p_s)\bn\cdot\bn&=&p_d ~~~~~~~~~~~~~~~~ \text{on }\Gamma,
\\\label{problem-eq9}
-\Tensor(\bu_s,p_s)\bn\cdot\btau&=&\mu\mathbb{K}^{1/2}\bu_s\cdot\btau
 ~~~~ \text{on }\Gamma.
\end{eqnarray}
Condition $(\ref{problem-eq7})$ is the result of mass conservation across the interface, condition $(\ref{problem-eq8})$ represents the {\color{black}fact that} normal force on the interface is balance, and condition $(\ref{problem-eq9})$ is the Beavers-Joseph-Saffman interface condition, in which $\mu\geq 0$ is a parameter depending on the properties of the porous medium.

Next, we recall some notations for Sobolev space $\cite{Sobolev}$.
Let $K$ be a polygon in $\Real^2$, $H^m(K)$ stands for the Sobolev space. We denote by $\|\cdot\|_{m,K}$ and $|\cdot|_{m,K}$ the norm and semi-norm in $H^m(K)$, $m\geq 0$. When $m=0$, $H^0(K)$ coincides with $L^2(K)$ and we shall drop the subscript $K$ in the norm and semi-norm notations.

{\color{black}We define the space $H(\mbox{div};K)$ as follows.}
\begin{eqnarray*}
H(\mbox{div};K)=\{\bv:\bv\in [L^2(K)]^d, \nabla\cdot\bv\in L^2(K)\},
\end{eqnarray*}
with norm
\begin{eqnarray*}
\|\bv\|_{H(\mbox{div},K)}=(\|\bv||^2_K+\|\nabla\cdot\bv\|^2_K)^{\frac12}.
\end{eqnarray*}
We also define
\begin{eqnarray*}
L^2_0(K)=\{ q\in L^2(K):~\int_{K} ~q~dx=0\}.
\end{eqnarray*}
Then the function space for the velocity and the pressure are defined as
\begin{eqnarray*}
V:=\{\bv\in H(\mbox{div},\Omega), \bv|_{\Omega_s}\in H^1(\Omega_s),~\bv=\mathbf{0}~\mbox{on}~\Gamma_s,~\bv\cdot\bn_d=0~\mbox{on}~\Gamma_d\},
\end{eqnarray*}
and
\begin{eqnarray*}
M:&=&L^2_0(\Omega).
\end{eqnarray*}

{\color{black}Now we are ready to state} the weak formulation of the Stokes-Darcy problem $(\ref{problem-eq1})-(\ref{problem-eq9})$. Find $(\bu,p)\in V\times M$ such that
\begin{eqnarray}
\label{weak-form1}a(\bu,\bv)+b(\bv,p)&=&(\mathbf{f}_s,\bv)_{\Omega_s} \quad \forall~ \bv\in V,\\
\label{weak-form2}b(\bu,q)&=&(f_d,q)_{\Omega_d}  \quad \forall~ q\in M,
\end{eqnarray}
where
\begin{eqnarray*}
a(\bu,\bv)&=&2\nu(D(\bu),D(\bv))_{\Omega_s}+(\mathbb{K}^{-1}\bu,\bv)_{\Omega_d}+\mu\mathbb{K}^{\frac12}\langle \bu_s\cdot\btau,\bv_s\cdot\btau\rangle_{\Gamma},\\
b(\bv,q)&=&-(\nabla\cdot \bv,q)_{\Omega}.
\end{eqnarray*}

The existence and the uniqueness of the weak solutions have been proved in $\cite{well-posedness}$.

\section{Discretization}
In this section, we first introduce some basic definitions and preliminaries which will be used throughout the rest of this article. Then we construct numerical scheme for $(\ref{weak-form1})-(\ref{weak-form2})$.

\subsection{Notations for Partitions}
In {\color{black}what follows}, $\Omega_i$ refers to either $\Omega_s$ or $\Omega_d$,
and it is the same for the other symbols with subscript $i$.
Let $\T_{i,h}$ be the partition of $\Omega_i$. Denote by $\T_h$ the union of $\T_{s,h}$ and $\T_{d,h}$, where $\T_{s,h}$
is a WG-regular partition $\cite{Wang2014a}$ and $\T_{d,h}$ consists of triangles or rectangles. $T_s$ represents the element of $\T_{s,h}$ and $T_d$ represents the element of $\T_{d,h}$. Denote the edges in $\T_h$ by $\E_h$, and define $e_i$ the edges on $\partial T_i$. Let $\E^s_h$ be the set of all edges in $\T_h\cap(\Omega_s\cup\Gamma_s)$, and $\E^d_h$ be the set of edges in $\T_h\cap (\Omega_d\cup \Gamma_d)$. The set of all edges in $\T_h\cap \Gamma$ is denoted by $\Gamma_h$. Especially, the partition $\T_{s,h}$ and $T_{d,h}$ are not necessary to be consistent on the interface $\Gamma$. Denote the size of $T_i$ by $h_{T_i}$, the mesh size of $\T_{i,h}$ by $h_i$. In addition, denote by $\rho\in P_{k_i}(T_{i})$ that $\rho|_{T_{i}}$ is polynomial with degree no more than $k_i$.

{\color{black}To define the WG method, we first give a brief introduction of weak function on $T_s$,
\begin{eqnarray}
\bv_{s,h}=\left\{\begin{array}{ll} \bv_{s,0}, \quad & \text{in} \quad T_s,
\\
\bv_{s,b},  & \text{on} \quad \partial T_s. \end{array}\right.
\end{eqnarray}
The weak function is formed by the internal function $\bv_{s,0}$ and the boundary function $\bv_{s,b}$, where $\bv_{s,b}$ may not necessarily be related to the trace of $\bv_{s,0}$ on $\partial {T_s}.$ Note that $\bv_{s,b}$ takes single value on $e_s$. For convenience, we write $\bv_{s,h}$ as $\{\bv_{s,0},\bv_{s,b}\}$ in short. }

In Stokes region, we define the following WG space for the velocity variable.
\begin{eqnarray*}
V^s_h&=&\{ \bv_{s,h}=\{\bv_{s,0},\bv_{s,b}\}\in [L^2(\Omega_s)]^2\times [L^2(\E_h^s)]^2:~\bv_{s,0}|_{T_s}\in [P_{\alpha_s}(T_s)]^2 ~\text{for} ~T_s\in \T_{s,h},
\\&&\bv_{s,b}|_{e_s}\in [P_{\beta}(e_s)]^2 ~\text{for}~e_s\in \E^s_h\cup \Gamma_h,~\bv_{s,b}|_{e_s}=0~\text{for}~{e_s}\in \E^s_h\cap\Gamma_s\},
\end{eqnarray*}
and the finite element space for the pressure variable as
\begin{eqnarray*}
M^s_h&=&\{ q_{s,h}\in L^2_0(\Omega_s): ~q_{s,h}|_{T_s}\in P_{\gamma _s}(T),~T_s\in \T_{s,h}\}
\end{eqnarray*}
where non-negative integers $\alpha_s,~\beta$ and $\gamma_s$  satisfy
\begin{eqnarray*}
\beta-1&\leq&\gamma_s\leq\beta\leq\alpha_s\leq\beta+1,\\
\alpha_s&\leq& \gamma_s+1,\\
1&\leq&\beta.
\end{eqnarray*}
\begin{remark}
For $\alpha_s=1,~\beta=0,~\gamma_s=0$, the situation is more complicated. Interested readers may refer to $\cite{wg-application1, Stokes2.2}$ for details.
\end{remark}

Then, we give the mixed finite element spaces corresponding to the Darcy region. For the velocity variable
\begin{eqnarray*}
V^d_h&=&\{ \bv_d\in H(\mbox{div},\Omega_d):~\bv_d|_T\in P_{\alpha_d}(T_d)~\text{for}~{T_d}\in\T_{d,h},~\bv_d\cdot\bn=0~\text{for}~\E^d_h\cap\Gamma_d\},
\end{eqnarray*}
and for the pressure variable
\begin{eqnarray*}
M^d_h&=&\{q_{d,h}\in L^2_0(\Omega_d):~q_{d,h}|_{T_d}\in P_{\gamma_d}(T_d)~\text{for}~T_d\in\T_{d,h}\},
\end{eqnarray*}
where
\begin{eqnarray*}
\gamma_d&\leq& \alpha_d,\\
\alpha_d-1&\leq&\gamma_d.
\end{eqnarray*}
We assume that $\nabla\cdot V^d_h\subset M^d_h$.

In order to impose the continuity of the velocity on the interface, we introduce the discrete space for Lagrange multiplier.
\begin{eqnarray*}
\Lambda_h=V^d_h\cdot\bn.
\end{eqnarray*}
Now, we can define the global discrete velocity space $V_h$ and the discrete pressure space $M_h$ as follows.
\begin{eqnarray*}
V_h &=&\{\bv_h=(\bv_{s,h},\bv_{d,h})\in V^s_h\times V^d_h: \sum_{e\in\Gamma_h}\int_e\eta (\bv_{s,h}-\bv_{d,h})\cdot\bn=0,~\forall ~\eta\in \Lambda_h\},\\
M_h&=&M^s_h\times M^d_h.
\end{eqnarray*}
\subsection{Discrete Weak Operators}
Next, we introduce some weak differential operators for $\bv_{s,h}\in V^s_h$.
\begin{definition}
For any $\bv_{s,h}\in V^s_h$, $T_s\in \T_{s,h}$,
the discrete weak gradient $\nabla_w\bv_{s,h}|_{T_s}
\in [P_{{\beta}}(T_s)]^{d\times d}$ satisfies
\begin{eqnarray}\label{vector-wgradient}
(\nabla_w\bv_{s,h},\tau)_{T_s}&=&-(\bv_{s,0},\nabla\cdot\tau)_{T_s}+
\langle \bv_{s,b},\tau\cdot\bn\rangle_{\partial T_s},\quad
\forall \tau\in[P_{\beta}(T_s)]^{d\times d}.
\end{eqnarray}
\end{definition}
Analogously, we can define the discrete weak divergence.
\begin{definition}
For any $\bv_{s,h}\in V^s_h$, $T_s\in \T_{s,h}$,
the discrete weak gradient  $\nabla_w\bv_{s,h}|_{T_s}
\in P_{{\beta}}(T_s)$ satisfies
\begin{eqnarray}\label{wdivergence}
(\nabla_w\cdot\bv_{s,h},q_{s,h})_{T}&=&-(\bv_{s,0},\nabla q_{s,h})_{T_s}+
\langle \bv_{s,b},q_{s,h}\bn\rangle_{\partial T_s},\quad
\forall q_{s,h}\in P_{\beta}(T_s).
\end{eqnarray}
\end{definition}

Finally, denote by $D_w(\bv_{s,h})$ the weak strain tensor given by
\begin{eqnarray*}
D_w(\bv_{s,h})=\frac12(\nabla_w\bv_{s,h}+\nabla_w\bv_{s,h}^T).
\end{eqnarray*}

\subsection{Numerical Scheme}
Define $Q_{h}=\{Q_{0},Q_{b}\}$ the projection operator from $L^2(\Omega_s)$ onto $V^s_h$, where
$Q_{0}$ is the $L^2$ projection onto
$[P_{\alpha_s}(T_s)]^2$, $\forall~ T_s\in \T_{s,h},$
$Q_{b}$ is the $L^2$ projection onto
$[P_{\beta}(e_s)]^2$, $\forall~ e_s\in \E^s_h$.

We are now in a position to give a numerical scheme for the coupled Stokes-Darcy problem. To this end, we define some bilinear forms in the discrete spaces. For any $\bu_h=(\bu_{s,h}, \bu_{d,h})$, $\bv_h=(\bv_{s,h}, \bv_{d,h}) \in {V}_h$, $p_h=(p_{s,h}, p_{d,h})$ and $q_h=(q_{s,h}, q_{d,h}) \in M_h$, define
\begin{eqnarray*}
a_{s,h}(\bu_{s,h},\bv_{s,h})&=&\sum_{T_s\in\T_{s,h}}(2\nu D_w(\bu_{s,h}),D_w(\bv_{s,h}))_{T_s}+s(\bu_{s,h},\bv_{s,h}),\\
s(\bu_{s,h},\bv_{s,h})&=&\sum_{T_s\in\T_{s,h}}h^{-1}_{T_s}\langle Q_b\bu_{s,0}-\bu_{s,b},Q_b\bv_{s,0}-\bv_{s,b}\rangle_{\partial T_s}, \\
a_{i,h}&=&\langle \mu \mathbb{K}^{-\frac12}\bu_{s,b}\cdot\bm{\tau},\bv_{s,b}\cdot\bm{\tau}\rangle_{\Gamma_h},\\
b_{s,h}(\bv_{s,h},q_{s,h})&=&-(\nabla_w\cdot\bv_{s,h},q_{s,h})_{\Omega_s},\\
b_{d,h}(\bv_{d,h},q_{d,h})&=&-(\nabla\cdot\bv_{d,h},q_{d,h})_{\Omega_d},\\
a_h(\bu_h,\bv_h)&=&a_{s,h}(\bu_{s,h},\bv_{s,h})+a_{i,h}(\bu_h,\bv_h)+a_d(\bu_{d,h},\bv_{d,h}),\\
b_h(\bv_h,q_h)&=&b_{s,h}(\bv_{s,h},q_{s,h})+b_{d,h}(\bv_{d,h},q_{d,h}).
\end{eqnarray*}
With these preparations, we give the numerical
scheme as follows.
\begin{algorithm1}
{\color{black}Seek} $\bu_h\in V_h$, $p_{h}\in M_{h}$ such that
\begin{eqnarray}\label{alg1}
a_h(\bu_h,\bv_h)+b_h(\bv_h,p_{h})&=&
(\mathbf{f}_s,\bv_{h})_{\Omega_s},
\\\label{alg2}
b_h(\bu_h,q_h)&=&(f_d,q_h)_{\Omega_d},
\end{eqnarray}
for all $\bv_h
=(\bv_{s,h},\bv_{d,h})\in V_h,$ and $q_h\in M_{h}$.
\end{algorithm1}

\section{Existence and Uniqueness}\label{existence}
In this section, we prove two important properties of the numerical scheme:
the boundedness of $a_h(\cdot,\cdot)$ and the inf-sup condition of $b_h(\cdot,\cdot)$. The existence and uniqueness of the approximate solutions then follow from the two properties.

We first define a discrete norm on $V^s_h$ by
\begin{eqnarray*}
&&\|\bv_h\|_{{V}^s_h}^2\\\nonumber
&=&2\nu\|D_w(\bv_{s,h})\|^2_{\Omega_s}+\sum_{T_s\in\T_{s,h}}h^{-1}_s\|Q_b\bv_{s,0}-\bv_{s,b}\|^2_{\partial T_s}+\|\mu^{\frac12}\mathbb{K}^{-\frac{1}{4}}\bv_{s,b}\cdot\btau\|^2_{\Gamma_{h}}.
\end{eqnarray*}
It is obvious that $\|\cdot\|_{V^s_h}$ is a semi-norm. In order to demonstrate $\|\cdot\|_{V^s_h}$ is a well-defined norm on ${V}^s_h$, we introduce the following estimate.

\begin{lemma}\label{lemma-korn}
For any $\bv_{s,h}\in V^s_h$, we have
\begin{eqnarray*}
\sumTa\|\nabla\bv_{_{s,0}}\|_{T_s}\le C\|\bv_{_{s,h}}\|_{V^s_h}, \quad \forall ~T_s\in\mathcal{T}_{s,h}.
\end{eqnarray*}
\end{lemma}
\begin{proof}
From $\cite{Brenner2003}$, we know the following discrete Korn's inequality holds.
\begin{eqnarray*}
&&\sumTa\|\nabla\bv_{_{s,0}}\|^2_{T_s}\\
&&\le C\left(\sumTa \|D(\bv_{_{s,0}})\|^2_{T_s}
+\sup_{\substack{\mathbf{m}\in RM,\|\mathbf{m}\|_{\Gamma_s}=1 \\ \int_{\Gamma_s}\mathbf{m} ds=\mathbf{0}}}
\left(\int_{\Gamma_s}\bv_{_{s,0}}\cdot\mathbf{m} ds\right)^2+\sum_{e_s\in\E_{_{s,h}\setminus \Gamma_s}}\|\pi_e[\bv_{_{s,0}}]\|_{e_s}^2\right),
\end{eqnarray*}
where $RM$ is the space of rigid motions, $\pi_{e_s}$ is the $L^2$ projection operator
onto $[P_1(e_s)]^d$, $[\cdot]$ denotes the jump on edges. 
Each term on the left hand of the inequality can be handled as follows.

Using the integration by parts and the definition of $\nabla_w$
on each element $T_s\in \T_{s,h}$, we have that
\begin{eqnarray*}
(D(\bv_{{s,0}}),D(\bv_{{s,0}}))_{T_s}
&=&(-\bv_{{s,0}},\nabla\cdot D(\bv_{{s,0}}))_{T_s}+\langle\bv_{{s,0}},D(\bv_{{s,0}})\cdot\bn
\rangle_{\partial T_s}
\\
&=&(-\bv_{{s,0}},\nabla\cdot D(\bv_{{s,0}}))_{T_s}+\langle\bv_{{s,b}},D(\bv_{{s,0}})\cdot\bn
\rangle_{\partial T_s}\\
&&+\langle\bv_{{s,0}}-\bv_{{s,b}},D(\bv_{{s,0}})\cdot\bn
\rangle_{\partial T_s}
\\
&=&(\nabla_w\bv_{s,h},D(\bv_{{s,0}}))_{T_s}+\langle Q_b\bv_{{s,0}}-\bv_{{s,b}},D(\bv_{{s,0}})\cdot\bn
\rangle_{\partial T_s}\\
&=&(D_w\bv_{s,h},D(\bv_{{s,0}}))_{T_s}+\langle Q_b\bv_{{s,0}}-\bv_{{s,b}},D(\bv_{{s,0}})\cdot\bn
\rangle_{\partial T_s}.
\end{eqnarray*}
{\color{black}Summing} over all element $T_s\in \mathcal{T}_{s,h}$ and {\color{black}applying} the trace inequality $(\ref{trace-thm})$,
the inverse inequality $(\ref{inverse-thm})$,
{\color{black}we obtain}
\begin{eqnarray*}
\|D(\bv_{{s,0}})\|^2_{T_s}&\le& C(\|D_w\bv_{s,h}\|_{T_s}\|D(\bv_{{s,0}})\|_{T_s}
+\|Q_b\bv_{{s,0}}-\bv_{{s,b}}\|_{\partial T_s}\|D(\bv_{{s,0}})\|_{\partial T_s})
\\
&\le& C(\|D_w\bv_{{s,h}}\|_{T_s}
+h_{T_s}^{-\frac12}\|Q_b\bv_{{s,0}}-\bv_{{s,b}}\|_{\partial T_s})  \|D(\bv_{{s,0}})\|_{T_s}.
\end{eqnarray*}
Therefore,
\begin{eqnarray*}
\sumTa \|D(\bv_{s,0})\|_{T_s}\le C\|\bv_{{s,h}}\|_{V^s_h}.
\end{eqnarray*}

For the second and the third terms, since $\beta\ge 1$ and $\bv_{s,b}=\mathbf{0}$ on $\Gamma_s$, 
we have
\begin{eqnarray*}
\sup_{\substack{\mathbf{m}\in RM,\|\mathbf{m}\|_{\Gamma_s}=1 \\ \int_{\Gamma_s}\mathbf{m} ds=\mathbf{0}}}
\left(\int_{\Gamma_s}\bv_{{s,0}}\cdot\mathbf{m} ds\right)
&=&\sup_{\substack{\mathbf{m}\in RM,\|\mathbf{m}\|_{\Gamma_s}=1 \\ \int_{\Gamma_s}\mathbf{m} ds=\mathbf{0}}}
\left(\int_{\Gamma_s}(Q_b\bv_{{s,0}}-\bv_{{s,b}})\cdot\mathbf{m} ds\right)
\\
&\le&C\|\bv_{_{s,h}}\|_{V^s_h},
\end{eqnarray*}
and
\begin{eqnarray*}
\sum_{e_s\in\E_{s,h}^0}\|\pi_e[\bv_{{s,0}}]\|_{e_s}
\le\sum_{e_s\in\E_{s,h}^0}\|Q_b[\bv_{{s,0}}]\|_{e_s}
\le\sumTa\|Q_b\bv_{{s,0}}-\bv_{{s,b}}\|_{\partial T_s}
&\le&C\|\bv_{{s,h}}\|_{V^s_h}.
\end{eqnarray*}
The proof is completed.
\end{proof}

\begin{lemma}
$\| \cdot\|_{V^s_h}$ provides a norm in $V^s_h$.
\end{lemma}

\begin{proof}
It suffices to check the positivity property of the semi-norm $\|\cdot\|_{V^s_h}$. To this end, assume that $\|\bv_{s,h}\|_{V^s_h}=0$ for some $\bv_{s,h}\in V^s_h$. Then we obtain $D_w(\bv_{s,h})=0$ on all $T_s\in\mathcal{T}_{s,h}$, $Q_{b}\bv_{s,0}=\bv_{s,b}$
on $\partial T_s$, $\bv_{s,b}\cdot{\bm\tau}=0$ on $\Gamma$. From the Lemma $\ref{lemma-korn}$, we have $\nabla \bv_{s,0}=0$ on all $T_s$,
which implies that $\bv_{s,0}=constant$ on every $T_s$. Moreover, $Q_b\bv_{s,0}=
\bv_{s,b}$ yields $\bv_{s,h}$ is a constant in $\Omega_s$. Combining with the fact that $\bv_{s,b}=0$ on $\Gamma_s$, we know that $\bv_{s,h}=0$.
\end{proof}

Now, we can define a discrete norm on $V_h$.
\begin{eqnarray}\label{def-tripbar}
\|\bv_h\|^2_{V_h}=\|\bv_{s,h}\|^2_{V^s_h}+\|\bv_{d,h}\|_{\Omega_d}^2+\|\nabla\cdot\bv_{d,h}\|_{\Omega_d}^2.
\end{eqnarray}

It follows from the definition of  norm (\ref{def-tripbar}) and the Cauchy Schwarz inequality that coercivity and boundedness hold true for the bilinear form $a_h(\cdot,\cdot)$.
\begin{lemma}
For any $\bu_h,\bv_h\in V_h$, we have
\begin{eqnarray}
\label{coe}a_h(\bv_h,\bv_h)&=&\|\bv_h\|^2_{V_h}, \qquad \forall \bv_h\in V_h, ~\nabla\cdot\bv_{d,h}=0,\\
\label{bound}|a_h(\bu_h,\bv_h)|&\le& C\|\bu_h\|_{V_h}\cdot\|\bv_h\|_{V_h}, \qquad\forall \bu_h,~\bv_h\in V_h.
\end{eqnarray}
\end{lemma}

Besides the projection $Q_h=\{Q_0,~Q_b\}$ defined in the previous section, we need another local $L^2$ projections, for each element $T_s\in \T_{s,h}$, denote by $\mathbf{Q}_h$ the $L^2$ projection onto $[P_{\beta}(T_s)]^{2\times 2}$ and by $\mathbb{Q}_h$ the $L^2$ projection onto $P_{\beta}(T_s)$.

\begin{lemma}
The projection operators defined above satisfy
\begin{eqnarray}
\label{comm1}\nabla_w(Q_h\bv)&=&\mathbf{Q}_h(\nabla \bv) \qquad  \forall ~\bv\in~[H^1(\Omega_s)]^d,\\
\label{comm2}\nabla_w\cdot(Q_h\bv)&=&\mathbb{Q}_h(\nabla\cdot\bv) \qquad  \forall ~\bv\in~H(\mbox{div},\Omega_s).
\end{eqnarray}
\end{lemma}
The proof of this Lemma can be found in $\cite{Wang2015b}$.

 As for Darcy region, denote the velocity space $V|_{\Omega_d}$ by $V^d$.
 Then we define the MFEM interpolant $\Pi^d_h:~V^d\cap [H^{\theta}(\Omega_d)]^2\rightarrow V^d_h$ with $\theta>0$ satisfying $\cite{MFE2}$, for any $\bv_d\in V^d\cap (H^{\theta}[\Omega_d)]^2,$
\begin{eqnarray}
\label{MFE-pro1}(\nabla\cdot\Pi^d_h\bv_d-\bv_d,q_{d,h})&=&0, \qquad \forall q_{d,h}\in M^d_h,\\
\label{MFE-pro2}\int_e((\Pi^d_h\bv_d-\bv_d)\cdot\bn_e)\bw_d\cdot\bn_e~ds&=&0, \qquad \forall e\in \Gamma^d_h\cup\Gamma_h,~\forall ~\bw_{d,h}\in V^d_h.
\end{eqnarray}

In addition, we denote by $R_h^s$ the $L^2$ projection onto $M^s_h$, and by $R^d_h$ the $L^2$ projection onto $M^d_h$.

Next, we introduce the discrete inf-sup condition for the bilinear form $b_h(\cdot,\cdot)$.
\begin{lemma}\label{inf-sup}(inf-sup) There exists a positive
constant $C$ independent of $h$ such that
\begin{eqnarray*}
\sup_{\bv_h\in V_h}\frac{b_h(\bv_h,q_h)}{\|\bv_h\|_{V_h}}
&\ge&C\|q_h\|_{M_h}
\end{eqnarray*}
for all $q_h\in M_h$.
\end{lemma}

\begin{proof}
According to $\cite{MR0365287}$, we know that for any $q_h\in M_h$,
there exists {\color{blue}a} $\bv\in [H_0^1(\Omega)]^2$ such that
\begin{eqnarray*}
\nabla\cdot\bv=-q_h \qquad in~\Omega,
\end{eqnarray*}
and
 $\|\bv\|_{1,\Omega}\le C\|q_h\|_{0,\Omega}$.

 Note that
 \begin{eqnarray*}
b_{s,h}(Q_h\bv,q_h)&=&-(\nabla_w\cdot Q_h\bv,q_h)_{\Omega_s}=-(\mathbb{Q}_h(\nabla\cdot \bv),q_h)_{\Omega_s}\\
&=&-(\nabla\cdot\bv,q_h)_{\Omega_s}=\|q_h\|^2_{\Omega_s},
\end{eqnarray*}
and
\begin{eqnarray*}
b_{d,h}(\bv,q_h)=-(\nabla\cdot\bv,q_h)_{\Omega_d}=\|q_h\|^2_{\Omega_d}.
\end{eqnarray*}
Next, we construct an projection operator $\pi_h:  ~(V\cap [H^1(\Omega)]^2)\rightarrow V_h$ such that
\begin{eqnarray*}
b_{s,h}(\pi_h\bv-Q_h\bv,q_h)=0,\qquad b_{d,h}(\pi_h\bv-\bv,q_h)=0,\quad \forall q_h\in M_h.
\end{eqnarray*}
Let $\pi_h\bv=(\pi^s_h\bv,\pi^d_h\bv)\in ~V^s_h\times V^d_h$. First, we take $\pi^s_h\bv=Q_h\bv$. {\color{black}It} is obvious that $b_{s,h}(\pi_h\bv-Q_h\bv,q_h)=0$. {\color{black}In addition}, the following estimate holds.
\begin{eqnarray*}
\|Q_h\bv_s\|_{V^s_h}\leq C\|\bv_s\|_{1,\Omega_s}
\end{eqnarray*}
Readers may refer to $\cite{WG1}$ for the proof for this estimate. Next, we need to define the operator $\pi^d_h\bv$. Consider the following auxiliary problem
\begin{eqnarray*}
\nabla\cdot\nabla \phi&=&0\qquad \mbox{in}~\Omega_d,\\
\nabla\phi\cdot\bn&=&0 \qquad \mbox{on} ~\Gamma_d,\\
\nabla\phi\cdot\bn&=&(\pi^s_h\bv-\bv)\cdot\bn\qquad \mbox{on}~\Gamma.
\end{eqnarray*}
It follows from the definition of the projection operator $Q_h$ that
\begin{eqnarray*}
\int_{\Gamma}(\pi^s_h\bv-\bv)\cdot\bn~ds=\int_{\Gamma}(Q_b\bv-\bv)\cdot\bn~ds=0.
\end{eqnarray*}
So the auxiliary problem is well-posed.
Let $\mathbf{z}=\nabla \phi$, we notice that the function $\pi^s_h\bv\cdot\bn\in H^{\theta}(\Gamma)$ for any $0\leq\theta\leq\frac12$.
By elliptic regularity $\cite{ell-reg}$,
\begin{eqnarray}\label{ell-err}
\label{regularity}\|\mathbf{z}\|_{\theta,\Omega_d}\leq C\|\pi^s_h \bv-\bv\|_{\theta-\frac12,\Gamma}\qquad 0\leq\theta\leq \frac12.
\end{eqnarray}
Let $\bw=\bv+\mathbf{z}$. Then we have
\begin{eqnarray}
\label{construct1}\nabla\cdot\bw&=&\nabla \cdot(\bv+\mathbf{z})=\nabla \cdot\bv \qquad ~~~~\mbox{in}~ \Omega_d ,\\
\label{construct2}\bw\cdot\bn&=&\bv\cdot\bn+\mathbf{z}\cdot\bn=\pi^s_h\bv\cdot\bn\qquad \mbox{on}~ \Gamma.
\end{eqnarray}
Define $\pi^d_h\bv:=\Pi^d_h\bw$. From the definition of $\Pi^d_h$, we know that
\begin{eqnarray*}
b_{d,h}(\pi^d_h\bv,q_{d,h})&=&b_{d,h}(\Pi^d_h\bw,q_{d,h})=b_{d,h}(\bw,q_{d,h})\\
&=&-(\nabla\cdot\bw,q_{d,h})=-(\nabla\cdot\bv,q_{d,h})=b_{d,h}(\bv,q_{d,h}), \qquad \forall ~q_{d,h}\in M^d_h.
\end{eqnarray*}
So the interpolant operator $\pi^d_h$ satisfies $b_{d,h}(\pi^d_h\bv-\bv,q_{d,h})=0$.

Next, we prove that $\pi_h\bv\in V_h$. For any $e\in \Gamma_h$ and $\eta\in \Lambda_h$, using $(\ref{MFE-pro2}),~(\ref{construct1})$ and $(\ref{construct2})$, we have
\begin{eqnarray*}
&&\int_e\pi^d_h\bv\cdot\bn\eta~ds=\int_e\Pi^d_h\bw\cdot\bn\eta~ds\\
&=&\int_e\bw\cdot\bn\eta~ds=\int_e\pi^s_h\bv\cdot\bn\eta~ds.
\end{eqnarray*}
It remains to give the bound of the operator $\pi^d_h$. From Lemma $(\ref{MFE-pro})$ and $(\ref{regularity})$
\begin{eqnarray*}
\|\pi^d_h\bv\|_{V^d_h}&=&\|\Pi^d_h\bw\|_{V^d_h}\\
&\leq& \|\Pi^d_h\bv\|_{V^d_h}+\|\Pi^d_h\mathbf{z}\|_{V^d_h}\\
&\leq& C(\|\bv\|_{1,\Omega_d}+\|\mathbf{z}\|_{\theta,\Omega_d})\\
&\leq& C(\|\bv\|_{1,\Omega_d}+\|(\pi^s_h\bv-\bv)\cdot\bn\|_{\Gamma}).
\end{eqnarray*}
Using the trace inequality $(\ref{trace-thm})$ and the projection inequality $(\ref{pro-est4})$, we have
\begin{eqnarray*}
\|(\pi^s_h\bv-\bv)\cdot\bn\|_e&\leq& \|Q_0\bv-\bv\|_e\\
&\leq& Ch^{-\frac12}\|Q_0\bv-\bv\|_{T_s}+Ch^{\frac12}\|\nabla(Q_0\bv-\bv)\|_{T_s}\\
&\leq& Ch^{\frac12}\|\bv\|_{1,T_s}.
\end{eqnarray*}
Thus, we obtain $\|\pi^d_h\bv\|_{V^d_h}\leq ~C\|\bv\|_{1,\Omega}$.
Furthermore,
\begin{eqnarray*}
\|\pi_h\bv\|_{V_h}\leq C\|\bv\|_{1,\Omega}.
\end{eqnarray*}
{\color{black}Combining} with the above estimates, we get
\begin{eqnarray*}
\frac{b_h(\pi_h\bv,q_h)}{\|\pi_h\bv\|_{V_h}}&=&\frac{b_{s,h}(Q_h\bv,q_h)+b_{d,h}(\Pi^d_h\bv,q_h)}{\|\pi_h\bv\|_{V_h}}\\
&\geq& C\frac{b_{s,h}(Q_h\bv,q_h)+b_{d,h}(\Pi^d_h\bv,q_h)}{\|\pi_h\bv\|_{1,\Omega}}\\
&\geq& C\frac{(\nabla\cdot\bv,q_h)+b_{d,h}(\Pi^d_h\bv,q_h)}{\|\pi_h\bv\|_{1,\Omega}}\\
&\geq& C\frac{\|q_h\|^2_{\Omega}}{\|\bv\|_{1,\Omega}}\\
&\geq& C\|q_h\|_{\Omega},
\end{eqnarray*}
which completes the proof.
\end{proof}
\begin{lemma}\label{projection}
For $\bv\in [H^1(\Omega)]^2,$ such that $\bv|_{\Omega_d}\in[H^{\gamma_d+2}(\Omega_i)]^2$, there exists $\tilde{\bv}_{h}\in V_h$ such that
\begin{eqnarray}
\label{interpolant3-pro} b_{d,h}(\bv-\tilde{\bv},q_{d,h})&=&0, \qquad \forall ~q_{d,h}\in M_h,\\
\label{interpolant3-err}\qquad\qquad\|\bv-\tilde{\bv}\|_{V^d_h}\leq& C& (h_d^{\alpha_d+1}|\bv|_{\alpha_d+1,\Omega_d}+h_d^{\gamma_d+1}|\nabla\cdot\bv|_{\gamma_d+1,\Omega_d}+h^{\alpha_s+\frac12}_s\|\bv\|_{\alpha_s+1,\Omega_s}).
\end{eqnarray}
\end{lemma}
\begin{proof}
Recall the interpolant $\pi^d_h\bv$ constructed in Lemma $\ref{inf-sup}$, then $(\ref{interpolant3-pro})$ can be deduced directly.
We only need to prove $(\ref{interpolant3-err})$. From the definition of $\pi^d_h$, we know
\begin{eqnarray}\label{ine-8}
\|\bv-\pi^d_h\bv\|_{V^d_h}=\|\bv-\Pi^d_h\bw\|_{V^d_h}\leq\|\bv-\Pi^d_h\bv\|_{V^d_h}+\|\Pi^d_h(\bw-\bv)\|_{V^d_h}.
\end{eqnarray}
Using Lemma $(\ref{MFE-pro})$, the first term on the right-hand side of $(\ref{ine-8})$ can be estimated as follows
\begin{eqnarray*}
\|\bv-\Pi^d_h\bv\|_{V^d_h}\leq C(h_d^{\alpha_d+1}|\bv|_{\alpha_d+1,\Omega_d}+h_d^{\gamma_d+1}|\nabla\cdot\bv|_{\gamma_d+1,\Omega_d}).
\end{eqnarray*}
For the second term, using estimate $(\ref{ell-err})$ and $(\ref{pro-est1})$,
\begin{eqnarray*}
\|\Pi^d_h(\bw-\bv)\|_{V^d_h}&=&\|\Pi^d_h\mathbf{z}\|_{V^d_h}\leq\|\mathbf{z}\|_{\theta,\Omega_d}\\
&\leq& \|(\pi^s_h\bv-\bv)\cdot\bn\|_{0,\Gamma}\leq Ch^{\alpha_s+1/2}_s\|\bv\|_{\alpha_s+1,\Omega_s}.
\end{eqnarray*}
Combining the estimates above we complete the proof.
\end{proof}

\begin{lemma}
The numercial scheme $(\ref{alg1})-
(\ref{alg2})$ has a {\color{black}unique} solution.
\end{lemma}
\begin{proof}
Since the problem is finite dimensional, it suffices to show that the solution is unique. Set $\mathbf{f}_s=\mathbf{0},~f_d=0.$ Then take $\bv_h=\bu_h$ and $q_h=p_h$, we have
\begin{eqnarray*}
a_h(\bu_h,\bu_h)&=&0,
\end{eqnarray*}
and
\begin{eqnarray*}
b_h(\bu_h,q_h)&=&0 \qquad \forall ~q_h\in ~M_h.
\end{eqnarray*}
Combining with the results above, we know that $a_h(\bu_h,\bu_h)=0$, which implies that $\bu_h=0$. Furthermore, we derive that
\begin{eqnarray*}
b(\bv_h,p_h)=0 \qquad \forall ~\bv_h\in V_h.
\end{eqnarray*}
From the inf-sup condition we know $p_h=0.$
\end{proof}

\section{Error Estimates}
In this section, we derive the optimal error estimates for the velocity in the energy norm and the pressure in the $L^2$ norm.

\begin{lemma}\label{pre-erreqn}
For any $\bw_s\in [H^1(\Omega_s)]^2,$ $\rho_s\in H^1(\Omega_s)$, and
 $\bv_{s,h}\in V_h^s$ , it
follows that
\begin{eqnarray}
&&\label{imp-equ1}(D_w(Q_h\bw_s), D_w(\bv_{s,h}))_{\Omega_s}\\\nonumber
&=&(D(\bw_s),D(\bv_{s,0}))
_{\Omega_s}-\sumTa
\langle \bv_{s,0}-\bv_{s,b}, \bQ_h D(\bw_{s})\cdot\bn\rangle_{\partial T_s},
\\
&&\label{imp-equ2}(\nabla_w\cdot\bv_{s,h}, R_h \rho_s)_{\Omega_s}\\\nonumber
&=&(\nabla\cdot\bv_{s,0},\rho_s)_{\Omega_s}
-\sumTa \langle \bv_{s,0}-\bv_{s,b}, (R_h\rho_s)\cdot\bn\rangle_{\partial T_s}.
\end{eqnarray}
\begin{proof}
According to the commutative property $(\ref{comm1})$, we know that $D_w(Q_h\bu_s)=\bQ_hD(\bu_s)$ is symmetric. Thus,
\begin{eqnarray*}
(D_w(Q_h\bw_s), D_w\bv_{s,h})_{T_s}=(\bQ_h D(\bw_s), D_w\bv_{s,h})_{T_s}=(\bQ_h D(\bw_s), \nabla_w\bv_{s,h})_{T_s}.
\end{eqnarray*}
It follows from the definition of weak gradient (\ref{vector-wgradient}) and the integration by parts, we have
\begin{eqnarray*}
&&\sum_{T_s\in\T_{s,h}}(D_w(Q_h\bw_s), \nabla_w\bv_{s,h})_{T_s}\\
&=&\sum_{T_s\in\T_{s,h}}(-(\nabla\cdot(\bQ_h D(\bw_s)),\bv_{s,0})_{T_s}+\langle \bv_{s,b},\bQ_h D(\bw_s)\bn\rangle_{\partial T_s})\\
&=&\sum_{T_s\in\T_{s,h}}((\bQ_h D(\bw_s),\nabla\bv_{s,0})_{T_s}-\langle \bv_{s,0}-\bv_{s,b},\bQ_h D(\bw_s)\bn\rangle_{\partial T_s})\\
&=&\sum_{T_s\in\T_{s,h}}((\bQ_h D(\bw_s),D(\bv_{s,0}))_{T_s}-\langle \bv_{s,0}-\bv_{s,b},\bQ_h D(\bw_s)\bn\rangle_{\partial T_s}).
\end{eqnarray*}
The proof of $(\ref{imp-equ2})$ is similar, so we omit details here.
\end{proof}
\end{lemma}

{\color{black}With the above lemma}, we can establish the error equations.
\begin{lemma}
Let $(\bu,p)$ be the solutions of $(\ref{problem-eq1})-(\ref{problem-eq9})$, and $(\bu_h,p_h)$ be the solutions of $(\ref{alg1})-
(\ref{alg2})$, we have
\begin{eqnarray}
\label{err-equation1}&&a_{s,h}(Q_h\bu_s-\bu_{s,h},\bv_{s,h})+a_{i,h}(Q_h\bu_s-\bu_{s,h},\bv_{s,h})+b_{s,h}(\bv_{s,h},R^s_hp_s-p_{s,h})\\\nonumber
&=&l_1(\bu_s,\bv_{s,h})-l_2(p_s,\bv_{s,h})-l_3(\bu_s,\bv_{s,h})-\langle p_d,\bv_{s,b}\cdot\bn\rangle_{\Gamma_h}+s(Q_h\bu_s,\bv_{s,h}),\\
\label{err-equation2}&&a_d(\bu_d-\bu_{d,h},\bv_{d,h})+b_d(\bv_{d,h},p_d-p_{d,h})=\langle p_d,\bv_{d,h}\cdot\bn\rangle_{\Gamma_h},\\
\label{err-equation3}&&b(Q_h\bu_s-\bu_{s,h},q_{s,h})=0,\\
\label{err-equation4}&&b(\bu_d-\bu_{d,h},q_{d,h})=0
\end{eqnarray}
for any $\bv\in V_h$ and $q_h\in M_h$,
where
\begin{eqnarray*}
l_1(\bu_s,\bv_{s,h})&=&\sum_{T_s\in\T_{s,h}}\langle 2\nu(\bv_{s,0}-\bv_{s,b}),
D(\bu_s)\cdot\bn-(\bQ_h D(\bu_s))\cdot\bn\rangle_{\partial T_s}\\
l_2(p_s,\bv_{s,h})&=&\sum_{T_s\in\T_{s,h}}\langle \bv_{s,0}-\bv_{s,b}, (p_s-R^s_h p_s)\bn\rangle_{\partial T_s}\\
l_3(\bu_s,\bv_{s,h})&=&\sum_{e\in\Gamma_h}\langle \mu\mathbb{K}^{-\frac12}(\bu_s-Q_b\bu_s)\cdot\btau,\bv_{s,b}\cdot\btau\rangle_e.
\end{eqnarray*}
\end{lemma}
\begin{proof}
Multiplying the Stokes equation $(\ref{problem-eq1})$ with $\bv_{s,0}$ in $\bv_{s,h}=\{\bv_{s,0},\bv_{s,b}\}\in V^s_h$ and integrating by parts over every element $T_s$.
\begin{eqnarray*}
(\mathbf{f}_s,\bv_{s,0})_{\Omega_s}&=&\sum_{T_s\in\T_{s,h}}(2\nu D(\bu_s),\nabla\bv_{s,0})_{T_s}
-\sum_{T_s\in\T_{s,h}}(\nabla\cdot\bv_{s,0},p_s)_{T_s}
\\&&
-\sum_{T_s\in\T_{s,h}}\langle2\nu\bv_{s,0},D(\bu_s)\cdot\bn\rangle_{\partial T_s}
+\sum_{T_s\in\T_{s,h}}\langle\bv_{s,0},p_s\bn\rangle_{\partial T_s}.
\end{eqnarray*}
By {\color{blue}the} regularity of the true solution $\bu_s$ and $p_s$, and the fact that $\bv_{s,b}=0$ on $\Gamma^s_h$,
\begin{eqnarray*}
(\mathbf{f}_s,\bv_{s,0})_{\Omega_s}&=&\sum_{T_s\in\T_{s,h}}(2\nu D(\bu_s),D(\bv_{s,0}))_{T_s}
-\sum_{T_s\in\T_{s,h}}(\nabla\cdot\bv_{s,0},p_s)_{T_s}
\\&&
-\sum_{T_s\in\T_{s,h}}\langle2\nu(\bv_{s,0}-\bv_{s,b}),D(\bu_s)\cdot\bn\rangle_{\partial T_s}
+\sum_{T_s\in\T_{s,h}}\langle\bv_{s,0}-\bv_{s,b},p_s\bn\rangle_{\partial T_s}
\\
&&-\sum_{e\in\Gamma_h}\langle\bv_{s,b},\Tensor(\bu_s,p_s)\bn\rangle_{e}.
\end{eqnarray*}
From the interface conditions, 
we know that
\begin{eqnarray*}
-\sum_{e\in\Gamma_h}\langle\bv_{s,b},\Tensor(\bu_s,p_s)\bn\rangle_{e}
=\langle p_d,\bv_{s,b}\cdot\bn\rangle_{\Gamma_h}+\langle \mu\mathbb{K}^{-\frac12}\bu_s\btau,\bv_{s,b}\cdot\btau\rangle_{\Gamma_h}.
\end{eqnarray*}
Applying Lemma $(\ref{pre-erreqn})$ yields
\begin{eqnarray*}
&&(\mathbf{f}_s,\bv_{s,0})_{\Omega_s}
\\
&=&\sumTa(2\nu D_w(Q_h\bu_s),D_w(\bv_{s,h}))_{T_s}-\sumTa(\nabla_w\cdot\bv_{s,h},R^s_hp_{_{s}})_{T_s}\\
&&-\sumTa\langle2\nu(\bv_{s,0}-\bv_{s,b}),D(\bu_s)\cdot\bn-(\bQ_hD(\bu_s))\cdot\bn\rangle_{\partial T_s}\\
&&+\sumTa\langle\bv_{s,0}-\bv_{s,b},(p_s-R^s_h p_s)\bn\rangle_{\partial T_s}
+\sum_{e\in\Gamma_h}\langle  p_d,\bv_{s,b}\cdot\bn\rangle_{e}\\
&&+\sum_{e\in\Gamma_h}\langle\mu\mathbb{K}^{-1/2} \bu_s\cdot\btau, \bv_{s,b}\cdot\btau\rangle_{e}
\\
&=&a_{s,h}(Q_h\bu_s,\bv_{s,h})+a_{i,h}(Q_h\bu_s,\bv_{s,h})+b_{s,h}(\bv_{s,h},R^s_hp_s)-s(Q_h\bu_s,\bv_{s,h})\\
&&-\sum_{T_s\in\T_{s,h}}\langle2\nu(\bv_{s,0}-\bv_{s,b}),D(\bu_s)\cdot\bn-(\bQ_hD(\bu_s))\cdot\bn\rangle_{\partial T_s}\\
&&+\sum_{T_s\in\T_{s,h}}\langle\bv_{s,0}-\bv_{s,b},(p_s-R^s_hp_s)\bn\rangle_{\partial T_s}
+\sum_{e\in\Gamma_h}\langle  p_d,\bv_{s,b}\cdot\bn\rangle_{e}\\
&&+\sum_{e\in\Gamma_h}\langle\mu\mathbb{K}^{-1/2} (\bu_s-Q_h\bu_s)\cdot\btau, \bv_{s,b}\cdot\btau\rangle_{e}.
\end{eqnarray*}
Therefore, we have
\begin{eqnarray*}
&&a_{s,h}(Q_h\bu_s,\bv_{s,h})+a_{i,h}(Q_h\bu_s,\bv_{s,h})+b_{s,h}(\bv_{s,h},R^s_hp_s)\\
&=&(\mathbf{f}_s,\bv_{s,0})_{\Omega_s}+s(Q_h\bu_s,\bv_{s,h})\\
 && \quad +\sum_{T\in\T_{s,h}}\langle2\nu(\bv_{s,0}-\bv_{s,b}),D(\bu_s)\cdot\bn-(\bQ_hD(\bu_s))\cdot\bn\rangle_{\partial T_s}\\
 && \quad -\sum_{T\in\T_{s,h}}\langle\bv_{s,0}-\bv_{s,b},(p_s-R^s_h p_s)\bn\rangle_{\partial T_s}
-\sum_{e\in\Gamma_h}\langle  p_d,\bv_{s,b}\cdot\bn\rangle_{e}\\
 && \quad -\sum_{e\in\Gamma_h}\langle\mu\mathbb{K}^{-1/2} (\bu_s-Q_h\bu_s)\cdot\btau, \bv_{s,b}\cdot\btau\rangle_{e}.
\end{eqnarray*}
Using the definition of $Q_h$ and $\mathbb{Q}_h$, we have
\begin{eqnarray*}
b_{s,h}(Q_h\bu,q_h)=-(\nabla_w\cdot(Q_h\bu),q_h)=-(\mathbb{Q}_h(\nabla\cdot\bu),q_h)=(\nabla\cdot\bu,q_h)=0.
\end{eqnarray*}
As for the Darcy's law $(\ref{problem-eq4})$, multiplying a test function $\bv_{d,h}\in V^d_h$ and using integration by parts on the Darcy region yields
\begin{eqnarray*}
0&=&(\mathbb{K}^{-1}\bu_d,\bv_{d,h})+(\nabla p_d,\bv_{d,h})\\
&=&(\mathbb{K}^{-1}\bu_d,\bv_{d,h})-(p_d,\nabla\cdot\bv_{d,h})-\langle p_d,\bv_{d,h}\cdot\bn\rangle_{\Gamma}\\
&=&a_{d,h}(\bu_d,\bv_{d,h})+b_{d,h}(\bv_{d,h},p_{d})-\langle p_d,\bv_{d,h}\cdot\bn\rangle_{\Gamma_h}
,
\end{eqnarray*}
which means that
\begin{eqnarray*}
a_{d}(\bu_d,\bv_{d,h})+b_{d,h}(\bv_{d,h},p_{d})=\langle p_d,\bv_{d,h}\cdot\bn\rangle_{\Gamma_h}.
\end{eqnarray*}
It is obvious that
\begin{eqnarray*}
b_{d,h}(\bu_{d},q_h)=(f_d,q_h).
\end{eqnarray*}
Combining with $(\ref{alg1})-(\ref{alg2})$, we obtain equations $(\ref{err-equation1})-(\ref{err-equation4})$.
\end{proof}

\begin{theorem}\label{th6-1}
Let $(\bu,p)$ be the solutions of the coupled problem $(\ref{problem-eq1})-(\ref{problem-eq9})$. Assume that $\bu|_{\Omega_i}$$\in [H^{\alpha_i+1}(\Omega)]^{2}$, $p|_{\Omega_i}\in H^{\gamma _i+1}(\Omega_s)$, $i=s,d$.  Let $(\bu_h,p_h)$ be the discrete solutions of $(\ref{alg1})-(\ref{alg2})$. Then the following estimate holds.
\begin{eqnarray}\label{err-est}
&&\|Q_h\bu_s-\bu_{s,h}\|_{V^s_h}+\|\bu_d-\bu_{d,h}\|_{V^d_h}\\\nonumber
&\leq& C(h_s^{\beta+1}\|\bu_s\|_{\beta+2,\Omega_s}+h_s^{\gamma_s+1}\|p_s\|_{\gamma_s+1,\Omega_s}+h_s^{\beta+1}\|\bu_s\|_{\beta+1,\Gamma}+h_s^{\alpha_s}\|\bu_s\|_{\alpha_s+1,\Omega_s})\\\nonumber
&+&C(h^{\alpha_d}\|\bu_d\|_{\alpha_d+1}+h^{\gamma_d+1}\|\bu_d\|_{\gamma_d+2,\Omega_d}+h_d^{\gamma_d+1/2}h^{1/2}_s\|p_d\|_{\gamma_d+1,\Omega_d}).
\end{eqnarray}
\end{theorem}

\begin{proof}
Adding equation $(\ref{err-equation2})$ to $(\ref{err-equation1})$, we have
\begin{eqnarray}\label{err-est1}
&&a_{s,h}(Q_h\bu_s-\bu_{s,h},\bv_{s,h})+b_{s,h}(\bv_{s,h},R^s_hp_s-p_{s,h})\\\nonumber
\quad&+&a_{i,h}(Q_h\bu_s-\bu_{s,h},\bv_{s,h})+a_{d}(\tilde{\bu}_d-\bu_{d,h},\bv_{d,h})+b_{d}(\bv_{d,h},R^d_hp_d-p_{d,h})\\\nonumber
&=&l_1(\bu_s,\bv_{s,h})-l_2(p_s,\bv_{s,h})-l_3(\bu_s,\bv_{s,h})+s(Q_h\bu_s,\bv_{s,h})\\\nonumber
&+&a_{d}(\tilde{\bu}_d-\bu_d,\bv_{d,h})+b_{d}(\bv_{d,h},R^d_hp_d-p_d)-\langle p_d,(\bv_{s,b}-\bv_{d,h})\cdot\bn\rangle_{\Gamma_h}.
\end{eqnarray}
From the Lemma $\ref{projection}$ and equation $(\ref{err-equation4})$, we get
\begin{eqnarray*}
b_d(\bu_{d,h}-\tilde{\bu}_d,q_{d,h})&=&0, \qquad \forall q_h\in M^d_h.
\end{eqnarray*}
 Since $\nabla \cdot V^d_h\subset M^d_h$,
 \begin{eqnarray*}
 \nabla\cdot(\bu_{d,h}-\tilde{\bu}_d)=0, ~in ~\Omega_d.
 \end{eqnarray*}
 Define $\be_{s,h}=Q_h\bu_s-\bu_{s,h},~\be_{d,h}=\tilde{\bu}_d-\bu_{d,h},~\epsilon_{s,h}=R^s_hp_s-p_{s,h}$ and $\epsilon_{d,h}=R^d_hp_d-p_{d,h}$.
 Taking $\bv_{s,h}=\be_{s,h},~\bv_{d,h}=\be_{d,h},~q_{s,h}=\epsilon_{s,h}$ and $q_{d,h}=\epsilon_{d,h}$ in $(\ref{err-est1})$, and combining with $(\ref{err-equation3})$, we have
 \begin{eqnarray*}
&&a_{s,h}(\be_{s,h},\be_{s,h})+a_{i,h}(\be_{s,h},\be_{s,h})+a_{d}(\be_{d,h},\be_{d,h})\\
&=&l_1(\bu_s,\be_{s,h})-l_2(p_s,\be_{s,h})-l_3(\bu_s,\be_{s,h})+s(Q_h\bu_s,\be_{s,h})\\
&+&a_{d}(\tilde{\bu}_d-\bu_d,\be_{d,h})-\langle p_d,(\bv_{s,b}-\bv_{d,h})\cdot\bn\rangle_{\Gamma_h}.
\end{eqnarray*}
We define $\be_h=(\be_{s,h},\be_{d,h})$. Making use of coercivity $(\ref{coe})$ and noting that $\nabla\cdot\be_{d,h}=0$ in $\Omega_d,$ we obtain
\begin{eqnarray*}
\|\be_h\|^2_{V_h}&=&a_{s,h}(\be_{s,h},\be_{s,h})+a_{i,h}(\be_{s,h},\be_{s,h})+a_{d}(\be_{d,h},\be_{d,h})\\
&=&l_1(\bu_s,\be_{s,h})-l_2(p_s,\be_{s,h})-l_3(\bu_s,\be_{s,h})+s(Q_h\bu_s,\be_{s,h})\\
&+&a_{d}(\tilde{\bu}_d-\bu_d,\be_{d,h})-\langle p_d,(\be_{s,b}-\be_{d,h})\cdot\bn\rangle_{\Gamma_h}.
\end{eqnarray*}

Next, we are going to estimate each term on the right-hand side of the above equation one by one.
It follows from $(\ref{l_1})-(\ref{stab})$ that
\begin{eqnarray*}
&&l_1(\bu_s,\be_{s,h})-l_2(p_s,\be_{s,h})-l_3(\bu_s,\be_{s,h})+s(Q_h\bu_s,\be_{s,h})\\
&\leq& C(h_s^{\beta+1}\|\bu_s\|_{\beta+2,\Omega_s}+h_s^{\gamma_s+1}\|p_s\|_{\gamma_s+1,\Omega_s}+h_s^{\beta+1}\|\bu_s\|_{\beta+1,\Gamma}+h_s^{\alpha_s}\|\bu_s\|_{\alpha_s+1,\Omega_s})\|\be_{s,h}\|_{V^s_h}.
\end{eqnarray*}

Using the Cauchy Schwarz inequality and $(\ref{interpolant3-err})$, we have
\begin{eqnarray*}
&&a_{d}(\tilde{\bu}_d-\bu_d,\be_{d,h})\\
&\leq&C\|\tilde{\bu}_d-\bu_d\|_{V^d_h}\cdot\|\be_{d,h}\|_{V^d_h}\\
&\leq&  C (h_d^{\alpha_d}\|\bu_d\|_{\alpha_d+1,\Omega_d}+h_d^{\gamma_d+1}\|\bu_d\|_{\gamma_d+2}+h^{\alpha_s+\frac12}_s\|\bu_s\|_{\alpha_s+1,\Omega_s})\|\be_{d,h}\|_{V^d_h}.
\end{eqnarray*}

Finally, to estimate $\langle p_d,(\be_{s,b}-\bv_{d,h})\cdot\bn\rangle_{\Gamma_h}$, we define a $L^2$ projection $R^e_h$ onto $\Lambda_h$ as follows.
\begin{eqnarray*}
\langle p_d,\lambda_h\rangle_{\Gamma_h}=\langle R^e_hp_d,\lambda_h\rangle_{\Gamma_h},\qquad \forall\lambda_h\in \Lambda_h.
\end{eqnarray*}
Since $\be_h=(\be_{s,h},\be_{d,h})\in ~V_h$, from the definition of $V_h$, we know that
\begin{eqnarray*}
\sum_{e\in\Gamma_h}\int_e\eta (\be_{s,b}-\be_{d,h})\cdot\bn=0, ~\forall \eta\in \Lambda_h.
\end{eqnarray*}
Combining with the fact that $R^e_hp_d\in \Lambda_h$,
\begin{eqnarray*}
\sum_{e\in\Gamma_h}\langle p_d,(\be_{s,b}-\be_{d,h})\cdot\bn\rangle_e=\sum_{e\in\Gamma_h}\langle p_d-R^e_hp_d,(\be_{s,b}-\be_{d,h})\cdot\bn\rangle_e
\end{eqnarray*}
Noting that $\be_{d,h}\cdot\bn\in \Lambda_h$, so we have
\begin{eqnarray*}
\sum_{e\in\Gamma_h}\langle p_d,(\be_{s,b}-\be_{d,h})\cdot\bn\rangle_e=\sum_{e\in\Gamma_h}\langle p_d-R^e_hp_d,(\be_{s,b})\cdot\bn\rangle_e
\end{eqnarray*}
For any constant vector $\mathbf{c}_e$, using the property of $R^e_h$, the trace inequality $(\ref{trace-thm})$ and Lemma $(\ref{lemma-korn})$, we obtain
\begin{eqnarray*}
&&\sum_{e\in\Gamma_h}\langle p_d-R^e_hp_d,(\be_{s,b})\cdot\bn\rangle_e\\
&=&\sum_{e\in\Gamma_h}\langle p_d-R^e_hp_d,(\be_{s,b}-\mathbf{c}_e)\cdot\bn\rangle_e\\
&\leq&\sum_{e\in\Gamma_h}\|p_d-R^e_hp_d\|_{e}\|\be_{s,b}-\mathbf{c}_e\|_e\\
&\leq&\sum_{e\in\Gamma_h}\|p_d-R^d_hp_d\|_{e}\|\be_{s,b}-\mathbf{c}_e\|_e\\
&\leq& Ch^{\gamma_d+1/2}_d\|p_d\|_{\gamma_d+1,\Omega_d}\sum_{T_s\in\T_{s,h}}(\|\be_{s,b}-Q_b\be_{s,0}\|_{\partial T_s}+\|Q_b\be_{s,0}-\mathbf{c}_e\|_{\partial T_s})\\
&\leq& Ch^{\gamma_d+1/2}_d\|p_d\|_{\gamma_d+1,\Omega_d}\Big(h^{1/2}_s\|\be_{s,h}\|_{V^s_h} +\sum_{T_s\in\T_{s,h}}C(h^{-1/2}_s\|\be_{s,0}-\mathbf{c}_e\|_{T_s}+h^{1/2}_s\|\nabla \be_{s,0}\|_{T_s})
   \Big)\\
&\leq&Ch^{\gamma_d+1/2}_d\|p_d\|_{\gamma_d+1,\Omega_d}h^{1/2}_s\|\be_{s,h}\|_{V^s_h}.
\end{eqnarray*}

Combining the above estimates, we obtain
\begin{eqnarray*}
&&\|\be_{s,h}\|_{V^s_h}+\|\bu_d-\bu_{d,h}\|_{V^d_h}\\
&=&\|\be_{s,h}\|_{V^s_h}+\|\be_{d,h}\|_{V^d_h}+\|\bu_d-\tilde{\bu}_d\|_{V^d_h}\\
&\leq& C(h_s^{\beta+1}\|\bu_s\|_{\beta+2,\Omega_s}+h_s^{\gamma_s+1}\|p_s\|_{\gamma_s+1,\Omega_s}+h_s^{\beta+1}\|\bu_s\|_{\beta+1,\Gamma}+h_s^{\alpha_s}\|\bu_s\|_{\alpha_s+1,\Omega_s})\\
&+&C(h^{\alpha_d}\|\bu_d\|_{\alpha_d+1,\Omega_d}+h_d^{\gamma_d+1}\|\bu_d\|_{\gamma_d+2}+h_d^{\gamma_d+1/2}h^{1/2}_s\|p_d\|_{\gamma_d+1,\Omega_d}).
\end{eqnarray*}
{\color{black}which completes} the proof of the theorem.
\end{proof}

\begin{theorem}\label{th6-2}
Under the assumption of Theorem $(\ref{th6-1})$, we have
\begin{eqnarray}\label{err-est2}
&&\|R^s_hp_s-p_{s,h}\|_{\Omega_s}+\|p_d-p_{h,d}\|_{\Omega_d}\\\nonumber
&\leq& C(h_s^{\beta+1}\|\bu_s\|_{\beta+2,\Omega_s}+h_s^{\gamma_s+1}\|p_s\|_{\gamma_s+1,\Omega_s}+h_s^{\beta+1}\|\bu_s\|_{\beta+1,\Gamma}+h_s^{\alpha_s}\|\bu_s\|_{\alpha_s+1,\Omega_s})\\\nonumber
&+&C(h^{\alpha_d}\|\bu_d\|_{\alpha_d+1,\Omega_d}+h_d^{\gamma_d+1}\|\bu_d\|_{\gamma_d+2}+h_d^{\gamma_d+1/2}h^{1/2}_s\|p_d\|_{\gamma_d+1,\Omega_d}).
\end{eqnarray}
\end{theorem}

\begin{proof}
The error equation $(\ref{err-est1})$ can be written as
\begin{eqnarray*}
&&b_{s,h}(\bv_{s,h},R^s_hp_s-p_{s,h})+b_{d}(\bv_{d,h},R^d_hp_d-p_{d,h})\\
&=&-a_{s,h}(Q_h\bu_s-\bu_{s,h},\bv_{s,h})-a_{i,h}(Q_h\bu_s-\bu_{s,h},\bv_{s,h})+a_{d}(\bu_{d,h}-\bu_d,\bv_{d,h})+l_1(\bu_s,\bv_{s,h})\\
&-&l_2(p_s,\bv_{s,h})-l_3(\bu_s,\bv_{s,h})+s(Q_h\bu_s,\bv_{s,h})+b_{d}(\bv_{d,h},R^d_hp_d-p_d)-\langle p_d,(\bv_{s,b}-\bv_{d,h})\cdot\bn_s\rangle_{\Gamma_h}.
\end{eqnarray*}
From the definition of $R^d_h$, we know that
\begin{eqnarray*}
b_{d}(\bv_{d,h},R^d_hp_d-p_d)=0.
\end{eqnarray*}
{\color{black}Thus,}
\begin{eqnarray*}
&&b_{s,h}(\bv_{s,h},R^s_hp_s-p_{s,h})+b_{d,h}(\bv_{d,h},R^d_hp_d-p_{d,h})\\
&\leq& C\|Q_h\bu_s-\bu_{s,h}\|_{V^s_h}\|\bv_{s,h}\|_{V^s_h}+C\|\bu_{d,h}-\bu_d\|_{V^d_h}\|\bv_{d,h}\|_{V^d_h}\\
&+& C(h_s^{\beta+1}\|\bu_s\|_{\beta+2,\Omega_s}+h_s^{\gamma_s+1}\|p_s\|_{\gamma_s+1,\Omega_s}+h_s^{\beta+1}\|\bu_s\|_{\beta+1,\Gamma}+h_s^{\alpha_s}\|\bu_s\|_{\alpha_s+1,\Omega_s})\|\bv_{s,h}\|_{V^s_h}\\
&+&Ch^{\gamma_d+1/2}_dh^{1/2}_s\|p_d\|_{\gamma_d+1,\Omega_d}\|\bv_{d,h}\|_{V^d_h}.
\end{eqnarray*}
It follows from the inf-sup condition and the Theorem $\ref{th6-1}$ that
\begin{eqnarray*}
&&\|R^s_hp_s-p_{s,h}\|_{\Omega_s}+\|R^d_hp_d-p_{d,h}\|_{\Omega_d}\\
&\leq& C(h_s^{\beta+1}\|\bu_s\|_{\beta+2,\Omega_s}+h_s^{\gamma_s+1}\|p\|_{\gamma_s+1,\Omega_s}+h_s^{\beta+1}\|\bu_s\|_{\beta+1,\Gamma}+h_s^{\alpha_s}\|\bu_s\|_{\alpha_s+1,\Omega_s})\\\nonumber
&+&C(h^{\alpha_d}\|\bu_d\|_{\alpha_d+1,\Omega_d}+h_d^{\gamma_d+1}\|\bu_d\|_{\gamma_d+2}+h_d^{\gamma_d+1/2}h^{1/2}_s\|p_d\|_{\gamma_d+1,\Omega_d}).
\end{eqnarray*}
Finally, using the estimate $(\ref{p-pro})$, we have
\begin{eqnarray*}
&&\|R^s_hp_s-p_s\|_{\Omega_s}+\|p_d-p_{h,d}\|_{\Omega_d}\\
&\leq& C(h_s^{\beta+1}\|\bu_s\|_{\beta+2,\Omega_s}+h_s^{\gamma_s+1}\|p_s\|_{\gamma_s+1,\Omega_s}+h_s^{\beta+1}\|\bu_s\|_{\beta+1,\Gamma}+h_s^{\alpha_s}\|\bu_s\|_{\alpha_s+1,\Omega_s})\\
&+&C(h^{\alpha_d}\|\bu_d\|_{\alpha_d+1,\Omega_d}+h_d^{\gamma_d+1}\|\bu_d\|_{\gamma_d+2}+h_d^{\gamma_d+1/2}h^{1/2}_s\|p_d\|_{\gamma_d+1,\Omega_d}),
\end{eqnarray*}
which complete the proof.
\end{proof}

\section{Numerical Test}
\def\bm#1{{\mathbf #1}}
In this section, we use two examples to verify our theoretical results  on
    the WG-MFEM scheme for the Stokes-Darcy problem.

In the first {\color{black}example},  we solve the following coupled problem
   on $\{\Omega_s=(0,\pi)\times(0,\pi)\}\cup \{\Omega_d=(0,\pi)\times(-\pi,0)\}$ and the interface $\Gamma=(0,\pi)\times\{0\}$:
\def\p#1{\begin{pmatrix} #1 \end{pmatrix}}
 \begin{align}\label{e1}
   -\nabla \cdot( \nabla \bm u_s + \nabla^T \bm u_s) + \nabla p_s &= \bm f_s &&  \hbox{ in }  \Omega_s, \\
   -\nabla \cdot(   \nabla p_d ) &= 0 &&  \hbox{ in }  \Omega_d,  \label{e1-2} \\
     \label{b1}
   \p{ \bm u_s \cdot \bm n \\ (-\nabla \bm u_s - \nabla^T \bm u_s+ p_s I) \bm n \cdot \bm n \\
      (-\nabla \bm u_s - \nabla^T \bm u_s+ p_s I)\bm n \cdot \bm \tau } &=
   \p{ \bm u_d \cdot \bm n\\ p_d \\ \bm u_s \cdot \bm \tau} && \hbox{ on }   \Gamma,
\end{align}
with outside boundary conditions
 \begin{align*}
    \bm u_s  &= \p{ 2\sin y \cos y \cos x\\
        (\sin^2 y-2)\sin x } &  \hbox{ on } & \Gamma_s, \\
         \bm v_d  \cdot\bm n &= \p{ (e^{-y} -e^y) \cos x \\
        (e^{-y} -e^y) \sin x } \cdot \bm n   &  \hbox{ on } & \Gamma_d.
\end{align*}
The source functions in \eqref{e1}-\eqref{e1-2} are defined by
 \begin{align*}
   \bm f_s &= \p{\sin y \cos x (5\cos y+1) \\
     \sin x (- \cos^2y +\frac32 \sin^2y -1+ \cos y) },\\
     f_d&=0.
\end{align*}
The exact solutions are
  \begin{align} \label{s1} \begin{aligned}
    \bm u_s  &= \p{ 2\sin y \cos y \cos x\\
        (\sin^2 y-2)\sin x } &  \hbox{ in } & \Omega_s, \\
      p_s &= \sin x \sin y  &  \hbox{ in } & \Omega_s, \\
      \bm u_d  &= \p{ -(e^y-e^{-y})\cos x\\
        -(e^y+e^{-y})\sin x } &  \hbox{ in } & \Omega_d, \\
         p_d   &= (e^y - e^{-y}) \sin x  &  \hbox{ in } & \Omega_d. \end{aligned}
 \end{align}
On the interface,  \eqref{b1} is satisfied as
 \begin{align}
     \label{b11}  \begin{aligned} \p{
     2 \sin x \\0\\0}  &= \p{     2 \sin x \\0\\0} \end{aligned} \text{ on } \Gamma.
\end{align}
    We plot the velocity field ($\bm u_s$ \& $\bm u_d$) in Figure \ref{gvector}.

\begin{figure}[htb]\begin{center}\setlength\unitlength{1in}
    \begin{picture}(3,3.3)
 \put(0,0.1){\includegraphics[width=2.4in]{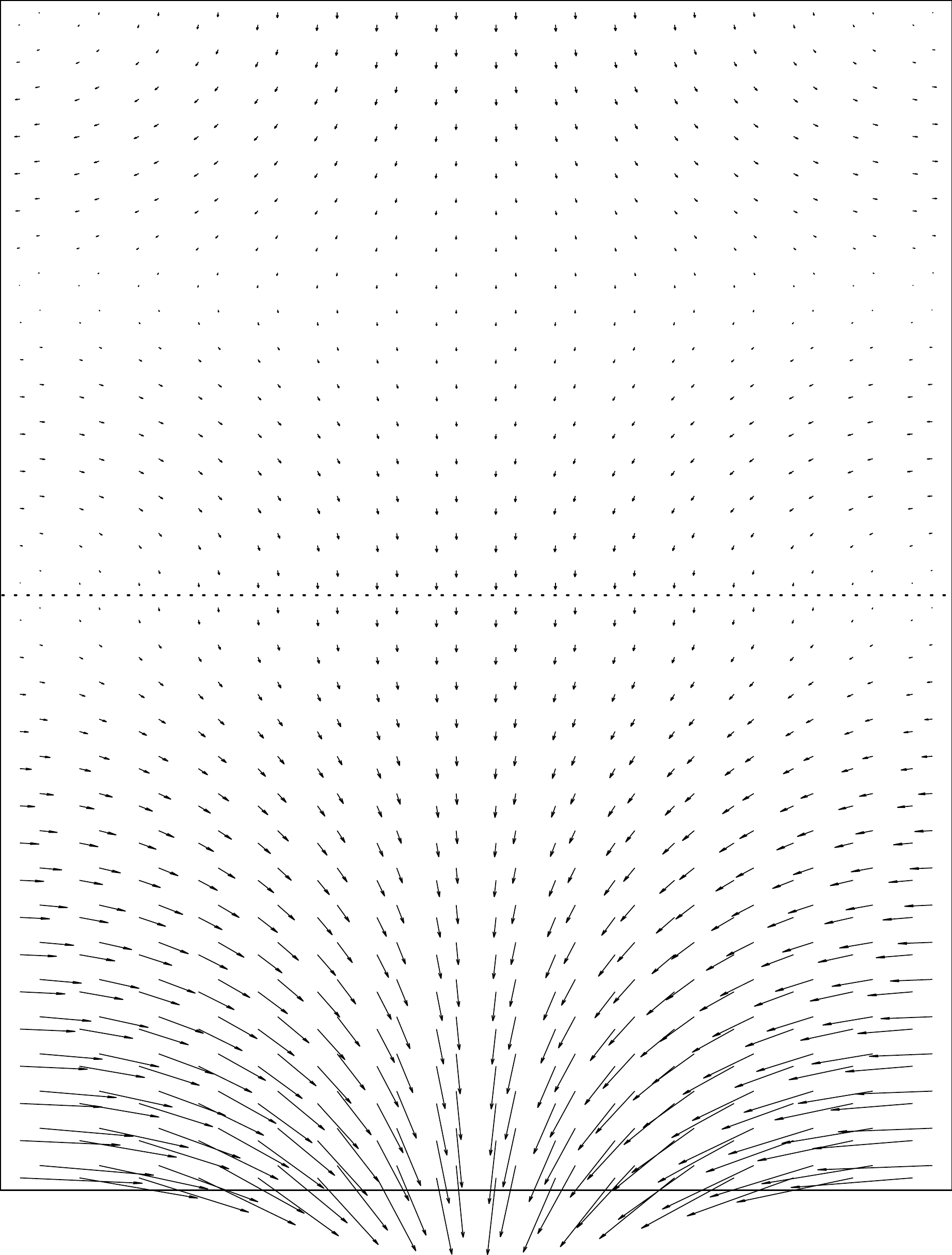}} 
    \end{picture}
\caption{ The velocity field of first example, \eqref{s1}.  } \label{gvector}
\end{center}
\end{figure}

In the computation, the first level grid consists of four triangles,  cutting each of two rectangles
   (see Figure \ref{gvector}) into two triangles by the north-west to south-east diagonal line.
Then, each subsequent grid is a bi-sectional refinement.
We apply the weak Galerkin $P_k$ finite element method for $\bm u_s$ and $p_s$ and the
   mixed BDM $P_k$ finite element method for computing $\bm u_d$ and $p_d$ in solving \eqref{s1}.
The errors and numerical orders of convergence for the unknown functions in various norms are reported in
  Tables \ref{t11}--\ref{t14}.
We can see that all numerical solutions are convergent of optimal order,
  as proved in our two theorems.
Because of coupling, the elliptic regularity for the Stokes-Darcy problem is not known.
Partially for this reason,  one order higher $L^2$ convergence for the velocity cannot
  be proved, or may be proved under some unknown conditions.
It does appear, for this example but not for next example,  in Tables \ref{t11}--\ref{t14}.
We still call such a phenomenon one-order superconvergence for the velocity in $L^2$
   by the WG-BDM $P_k$ elements ($k=1,2,3,4$).

\begin{table}[htb]
  \caption{ The errors and the order $O(h^k)$ of convergence by the $P_1$ WG
     elements and $BDM_1$ elements, for \eqref{s1}.
   \label{t11} }
\begin{center}  \begin{tabular}{c|rr|rr|rr}  
\hline
level & $ \|Q_0 \bm u_s- \bm u_{s,h}\|_0 $ & $k$ &  $ \trb Q_h \bm u_s- \bm u_{s,h} \trb $
    & $k$ & $\|R^s_h p_s-p_{s,h}\|_0 $ & $k$  \\ \hline
 4& 0.3643E+00&  1.8& 0.9974E+00&  1.0& 0.3770E+00&  1.1\\
 5& 0.9585E-01&  1.9& 0.5045E+00&  1.0& 0.1480E+00&  1.3\\
 6& 0.2437E-01&  2.0& 0.2523E+00&  1.0& 0.5991E-01&  1.3\\
 7& 0.6120E-02&  2.0& 0.1260E+00&  1.0& 0.2721E-01&  1.1\\
\hline
  & $ \|I_h \bm u_d-\bm u_{d,h}\|_0 $ & $k$ &
   \multicolumn{2}{r|}{ $ \|\operatorname{\mbox{div}} (\bm u_d-\bm u_{d,h})\|_0\; n $} &
    $ \|I_h p_d-p_{d,h}\|_0 $ &$k$  \\ \hline
 4& 0.9155E+00&  2.0& 0.5638E+01&  1.0& 0.6896E+00&  2.0\\
 5& 0.2302E+00&  2.0& 0.2840E+01&  1.0& 0.1731E+00&  2.0\\
 6& 0.5765E-01&  2.0& 0.1423E+01&  1.0& 0.4332E-01&  2.0\\
 7& 0.1442E-01&  2.0& 0.7117E+00&  1.0& 0.1083E-01&  2.0\\
      \hline
\end{tabular}\end{center}
\end{table}

\begin{table}[htb]
  \caption{ The errors and the order $O(h^k)$ of convergence by the $P_2$ WG
     elements and $BDM_2$ elements, for \eqref{s1}.
   \label{t12} }
\begin{center}  \begin{tabular}{c|rr|rr|rr}  
\hline
level & $ \|Q_0 \bm u_s- \bm u_{s,h}\|_0 $ & $k$ &  $ \trb Q_h \bm u_s- \bm u_{s,h} \trb $
    & $k$ & $\|R^s_h p_s-p_{s,h}\|_0 $ & $k$  \\ \hline
 3& 0.1269E+00&  2.3& 0.7194E+00&  1.7& 0.1979E+00&  1.7 \\
 4& 0.2179E-01&  2.5& 0.2250E+00&  1.7& 0.5040E-01&  2.0 \\
 5& 0.3200E-02&  2.8& 0.6331E-01&  1.8& 0.1163E-01&  2.1 \\
 6& 0.4273E-03&  2.9& 0.1664E-01&  1.9& 0.2697E-02&  2.1 \\
\hline
   & $ \|I_h \bm u_d-\bm u_{d,h}\|_0 $ & $k$ &
   \multicolumn{2}{r|}{ $ \|\operatorname{\mbox{div}} (\bm u_d-\bm u_{d,h})\|_0\; k $} &
    $ \|I_h p_d-p_{d,h}\|_0 $ &$k$  \\ \hline
 3& 0.1926E+00&  2.9& 0.1456E+01&  1.9& 0.1294E+01&  1.9 \\
 4& 0.2438E-01&  3.0& 0.3721E+00&  2.0& 0.3297E+00&  2.0 \\
 5& 0.3056E-02&  3.0& 0.9355E-01&  2.0& 0.8283E-01&  2.0 \\
 6& 0.3821E-03&  3.0& 0.2342E-01&  2.0& 0.2073E-01&  2.0 \\
      \hline
\end{tabular}\end{center}
\end{table}

\begin{table}[htb]
  \caption{ The errors and the order $O(h^k)$ of convergence by the $P_3$ WG
     elements and $BDM_3$ elements, for \eqref{s1}.
   \label{t13} }
\begin{center}  \begin{tabular}{c|rr|rr|rr}  
\hline
level & $ \|Q_0 \bm u_s- \bm u_{s,h}\|_0 $ & $k$ &  $ \trb Q_h \bm u_s- \bm u_{s,h} \trb $
    & $k$ & $\|R^s_h p_s-p_{s,h}\|_0 $ & $k$  \\ \hline
 3& 0.1791E-01&  3.5& 0.1887E+00&  2.7& 0.2979E-01&  2.6 \\
 4& 0.1342E-02&  3.7& 0.2765E-01&  2.8& 0.3425E-02&  3.1 \\
 5& 0.9197E-04&  3.9& 0.3746E-02&  2.9& 0.3748E-03&  3.2 \\
 6& 0.5967E-05&  3.9& 0.4836E-03&  3.0& 0.4302E-04&  3.1 \\
\hline
   & $ \|I_h \bm u_d-\bm u_{d,h}\|_0 $ & $k$ &
   \multicolumn{2}{r|}{ $ \|\operatorname{\mbox{div}} (\bm u_d-\bm u_{d,h})\|_0\; k $} &
    $ \|I_h p_d-p_{d,h}\|_0 $ &$k$  \\ \hline
 3& 0.1207E-01&  3.9& 0.1281E+00&  2.9& 0.1222E+00&  2.9 \\
 4& 0.7753E-03&  4.0& 0.1640E-01&  3.0& 0.1562E-01&  3.0 \\
 5& 0.4877E-04&  4.0& 0.2062E-02&  3.0& 0.1964E-02&  3.0 \\
 6& 0.3051E-05&  4.0& 0.2582E-03&  3.0& 0.2458E-03&  3.0 \\
      \hline
\end{tabular}\end{center}
\end{table}

\begin{table}[htb]
  \caption{ The errors and the order $O(h^k)$ of convergence by the $P_4$ WG
     elements and $BDM_4$ elements, for \eqref{s1}.
   \label{t14} }
\begin{center}  \begin{tabular}{c|rr|rr|rr}  
\hline
level & $ \|Q_0 \bm u_s- \bm u_{s,h}\|_0 $ & $k$ &  $ \trb Q_h \bm u_s- \bm u_{s,h} \trb $
    & $k$ & $\|R^s_h p_s-p_{s,h}\|_0 $ & $k$  \\ \hline
 3& 0.1925E-02&  4.7& 0.3302E-01&  3.7& 0.3850E-02&  3.9 \\
 4& 0.6334E-04&  4.9& 0.2141E-02&  3.9& 0.2499E-03&  3.9 \\
 5& 0.2007E-05&  5.0& 0.1349E-03&  4.0& 0.1579E-04&  4.0 \\
\hline
  & $ \|I_h \bm u_d-\bm u_{d,h}\|_0 $ & $k$ &
   \multicolumn{2}{r|}{ $ \|\operatorname{\mbox{div}} (\bm u_d-\bm u_{d,h})\|_0\; k $} &
    $ \|I_h p_d-p_{d,h}\|_0 $ &$k$  \\ \hline
 3& 0.6309E-03&  4.9& 0.8709E-02&  3.9& 0.7438E-02&  3.9 \\
 4& 0.1989E-04&  5.0& 0.5589E-03&  4.0& 0.4763E-03&  4.0 \\
 5& 0.6208E-06&  5.0& 0.3517E-04&  4.0& 0.2995E-04&  4.0 \\
      \hline
\end{tabular}\end{center}
\end{table}

In the second numerical example, we solve the coupled problem \eqref{s1}
   on domain $\{\Omega_s=(0,1)\times(1,2)\}\cup \{\Omega_d=(0,1)\times(0,1)\}$.
The exact solutions are
  \begin{align} \label{s2} \begin{aligned}
    \bm u_s  &= \p{-\cos (\pi x) \sin(\pi y)\\
                    \sin (\pi x) \cos (\pi y) } &  \hbox{ in } & \Omega_s,  \\
      p_s &= \sin (\pi x )   &  \hbox{ in } & \Omega_s, \\
    \bm u_d  &= \p{-y \pi \cos (\pi x) \\
                   -\sin (\pi x) } &  \hbox{ in } & \Omega_d,  \\
         p_d   &= y \sin (\pi x ) &  \hbox{ in } & \Omega_d. \end{aligned}
 \end{align}
On the interface $\Gamma=(0,1)\times\{1\}$, the  condition \eqref{b1} is reduced to
 \begin{align}
     \label{b12}  \begin{aligned} \p{
     - \sin \pi x \\ \sin \pi x \\0}  &= \p{  - \sin \pi x \\ \sin \pi x \\0} \end{aligned}
     \text{ on } \Gamma.
\end{align}
The velocity field $(\bm u_s,\bm u_d)$ is plotted in Figure \ref{gvector2}.
The computational grids are same as those in last example, described above.
We list the order of convergence in Tables \ref{t21}--\ref{t24}, by $P_1$, $P_2$, $P_3$ and $P_4$
  WG-BDM coupled finite element methods.
The results confirm the two theorems proved here.
Like the computation for the first example, one order superconvergence is obtained in the $P_1$ and $P_3$ WG-BDM element velocity solutions in $L^2$-norm.
Unlike the first example,  this example does not have an $L^2$-superconvergence
    for the $P_2$ and the $P_4$ WG--BDM coupled elements.
To see the superconvergence in the other cases,  we plot the solution and the error for the $P_3$ coupled
  element in Figures \ref{g-u1}--\ref{g-p}.

  \begin{figure}[htb]\begin{center}\setlength\unitlength{1in}
    \begin{picture}(2.0,3.4)
 \put(0,0){\includegraphics[width=2.4in]{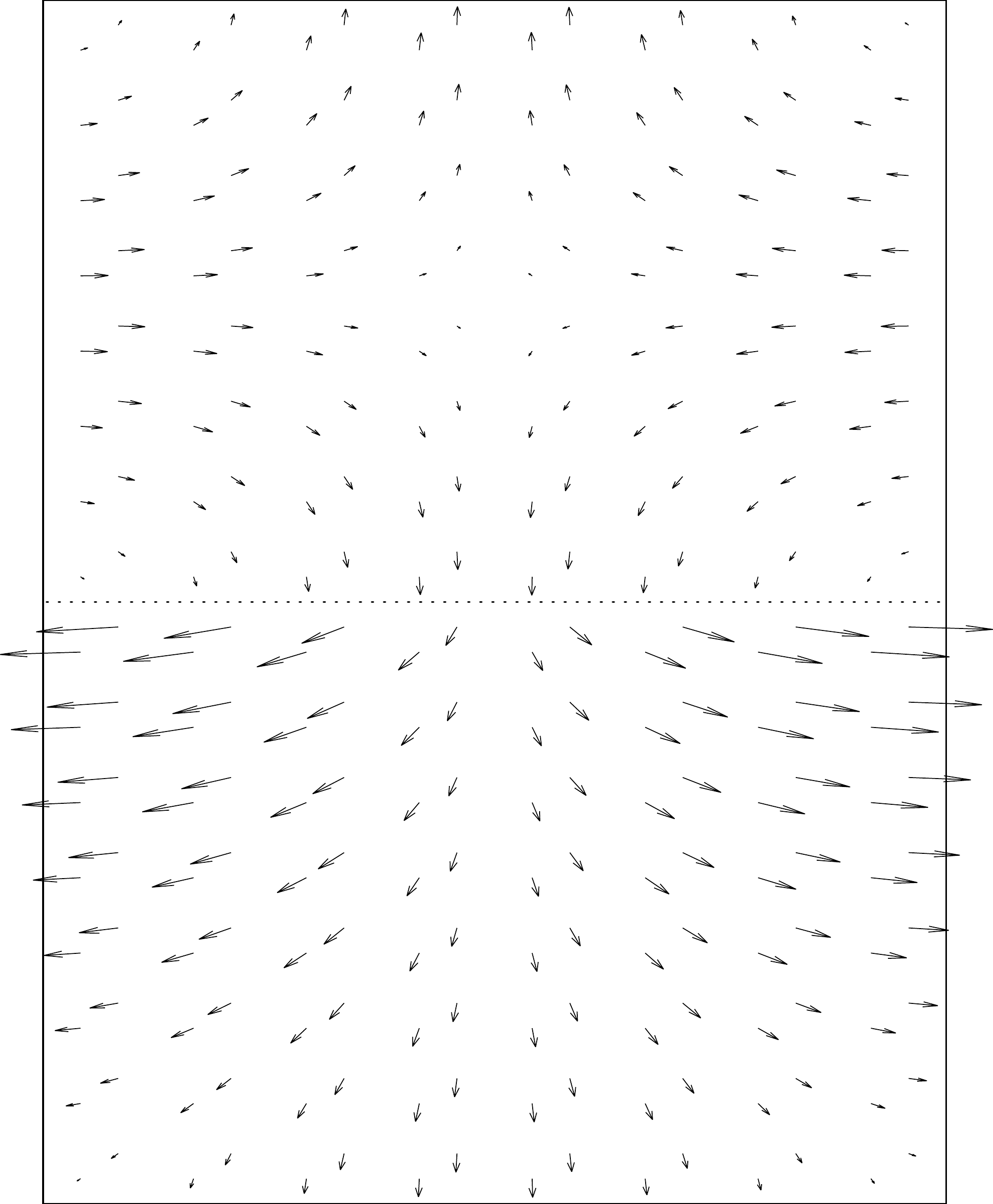}}
    \end{picture}
\caption{ The velocity field of the second example, \eqref{s2}.  } \label{gvector2}
\end{center}
\end{figure}

\begin{table}[htb]
  \caption{ The errors  and the order $O(h^k)$ of convergence by the
   $P_1$ WG and $BDM_1$ coupled element,  for \eqref{s2}.
   \label{t21} }
\begin{center}  \begin{tabular}{c|rr|rr|rr}  
\hline
level & $ \|Q_0 \bm u_s- \bm u_{s,h}\|_0 $ & $k$ &  $ \trb Q_h \bm u_s- \bm u_{s,h} \trb $
    & $k$ & $\|R^s_h p_s-p_{s,h}\|_0 $ & $k$  \\ \hline
 5& 0.9163E-02&  1.9& 0.2177E+00&  1.0& 0.6146E-01&  1.0 \\
 6& 0.2312E-02&  2.0& 0.1087E+00&  1.0& 0.3057E-01&  1.0 \\
 7& 0.5790E-03&  2.0& 0.5424E-01&  1.0& 0.1525E-01&  1.0 \\
\hline
   & $ \|I_h \bm u_d-\bm u_{d,h}\|_0 $ & $k$ &
   \multicolumn{2}{r|}{ $ \|\operatorname{\mbox{div}} (\bm u_d-\bm u_{d,h})\|_0\; k $} &
    $ \|I_h p_d-p_{d,h}\|_0 $ &$k$  \\ \hline
 5& 0.3463E-02&  2.0& 0.1960E+00&  1.0& 0.1458E-02&  2.0 \\
 6& 0.8565E-03&  2.0& 0.9811E-01&  1.0& 0.3651E-03&  2.0 \\
 7& 0.2128E-03&  2.0& 0.4907E-01&  1.0& 0.9126E-04&  2.0 \\
     \hline
\end{tabular}\end{center}
\end{table}

\begin{table}[htb]
  \caption{ The errors  and the order $O(h^k)$ of convergence by the
   $P_2$ WG and $BDM_2$ coupled element,  for \eqref{s2}.
   \label{t22} }
\begin{center}  \begin{tabular}{c|rr|rr|rr}  
\hline
level & $ \|Q_0 \bm u_s- \bm u_{s,h}\|_0 $ & $k$ &  $ \trb Q_h \bm u_s- \bm u_{s,h} \trb $
    & $k$ & $\|R^s_h p_s-p_{s,h}\|_0 $ & $k$  \\ \hline
 4& 0.1293E-01&  2.1& 0.8738E-01&  2.0& 0.2327E-01&  2.0 \\
 5& 0.3196E-02&  2.0& 0.2210E-01&  2.0& 0.5794E-02&  2.0 \\
 6& 0.7967E-03&  2.0& 0.5555E-02&  2.0& 0.1444E-02&  2.0 \\
\hline
  & $ \|I_h \bm u_d-\bm u_{d,h}\|_0 $ & $k$ &
   \multicolumn{2}{r|}{ $ \|\operatorname{\mbox{div}} (\bm u_d-\bm u_{d,h})\|_0\; k $} &
    $ \|I_h p_d-p_{d,h}\|_0 $ &$k$  \\ \hline
 4& 0.1192E-01&  2.0& 0.5139E-01&  2.0& 0.4266E-02&  2.0 \\
 5& 0.2995E-02&  2.0& 0.1292E-01&  2.0& 0.1070E-02&  2.0 \\
 6& 0.7497E-03&  2.0& 0.3235E-02&  2.0& 0.2678E-03&  2.0 \\
     \hline
\end{tabular}\end{center}
\end{table}

\begin{figure}[htb]\begin{center}\setlength\unitlength{1in}
    \begin{picture}(2.0,2.5)
 \put(0,1.6){\includegraphics[width=2.4in]{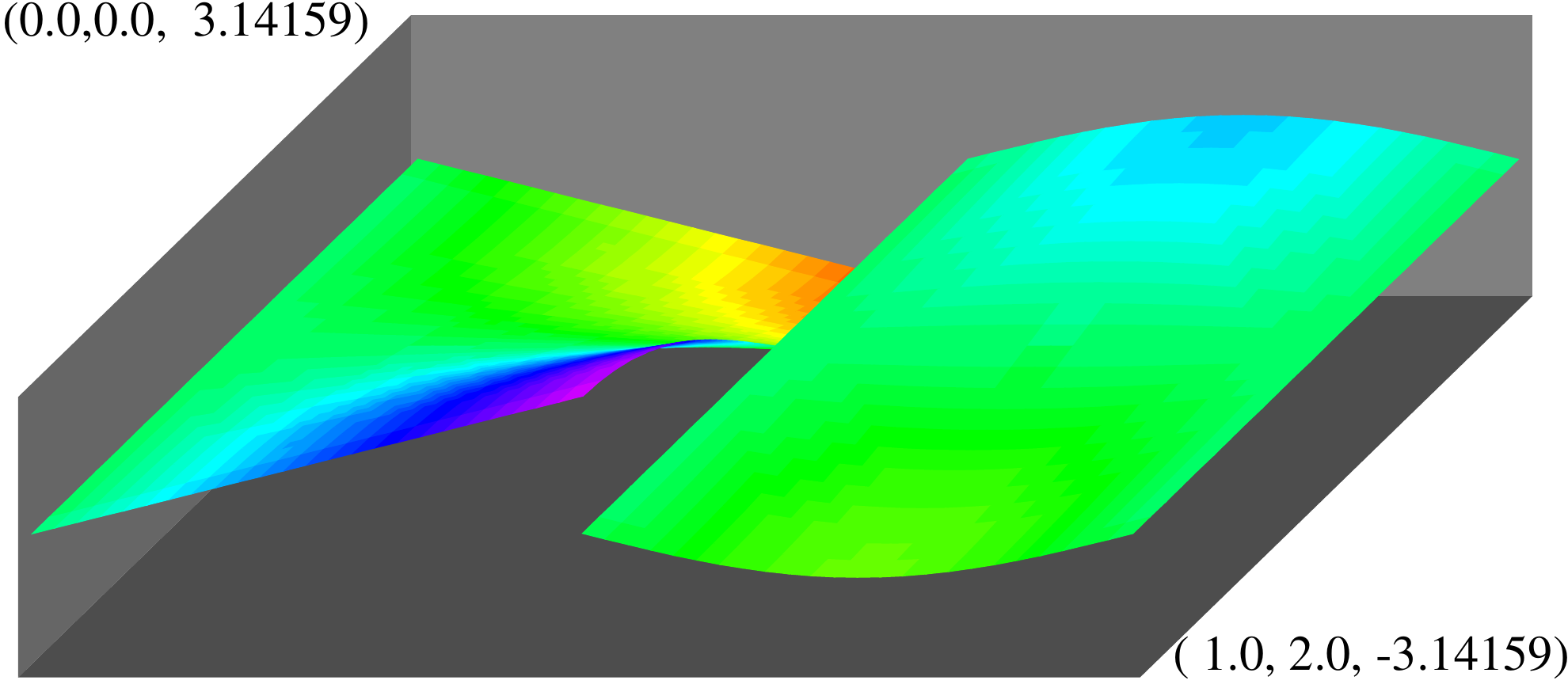}}
 \put(0,0){\includegraphics[width=2.4in]{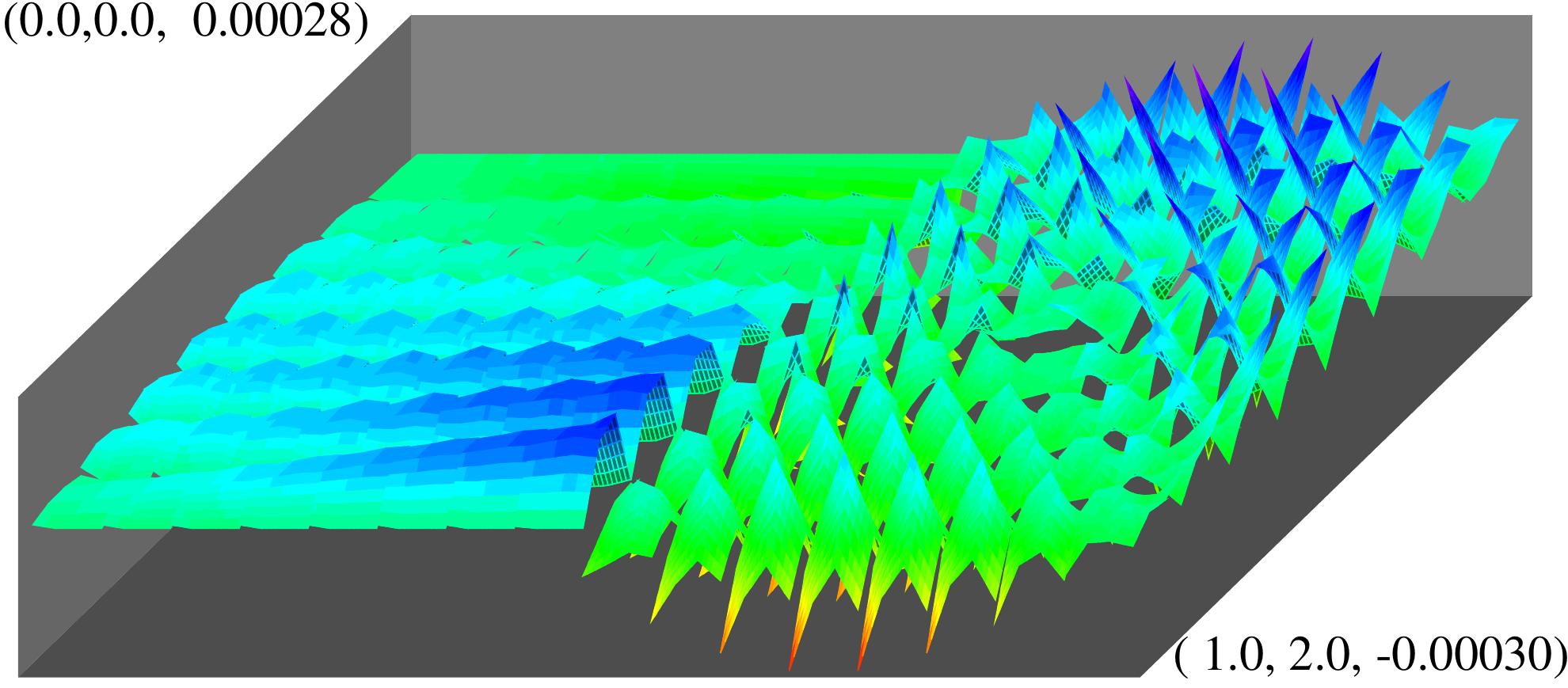}}
    \end{picture}
\caption{\label{g-u1} The solution for $(\bm u_s,\bm u_d)_1$ and its error by $P_3$ elements on level 4, for
    \eqref{s2}.  }
\end{center}
\end{figure}

\begin{table}[htb]
  \caption{ The errors  and the order $O(h^k)$ of convergence by the
   $P_3$ WG and $BDM_3$ coupled element,  for \eqref{s2}.
   \label{t23} }
\begin{center}  \begin{tabular}{c|rr|rr|rr}  
\hline
level & $ \|Q_0 \bm u_s- \bm u_{s,h}\|_0 $ & $k$ &  $ \trb Q_h \bm u_s- \bm u_{s,h} \trb $
    & $k$ & $\|R^s_h p_s-p_{s,h}\|_0 $ & $k$  \\ \hline
 3& 0.1074E-02&  3.9& 0.3643E-01&  2.9& 0.5218E-02&  3.0 \\
 4& 0.6950E-04&  3.9& 0.4757E-02&  2.9& 0.6587E-03&  3.0 \\
 5& 0.4417E-05&  4.0& 0.6074E-03&  3.0& 0.8273E-04&  3.0 \\
\hline
   & $ \|I_h \bm u_d-\bm u_{d,h}\|_0 $ & $k$ &
   \multicolumn{2}{r|}{ $ \|\operatorname{\mbox{div}} (\bm u_d-\bm u_{d,h})\|_0\; k $} &
    $ \|I_h p_d-p_{d,h}\|_0 $ &$k$  \\ \hline
 3& 0.1091E-02&  4.0& 0.1481E-01&  2.9& 0.1321E-02&  2.7 \\
 4& 0.6864E-04&  4.0& 0.1868E-02&  3.0& 0.1725E-03&  2.9 \\
 5& 0.4294E-05&  4.0& 0.2341E-03&  3.0& 0.2179E-04&  3.0 \\
     \hline
\end{tabular}\end{center}
\end{table}

\begin{figure}[htb]\begin{center}\setlength\unitlength{1in}
    \begin{picture}(2.0,2.0)
 \put(0,1.2){\includegraphics[width=2.2in]{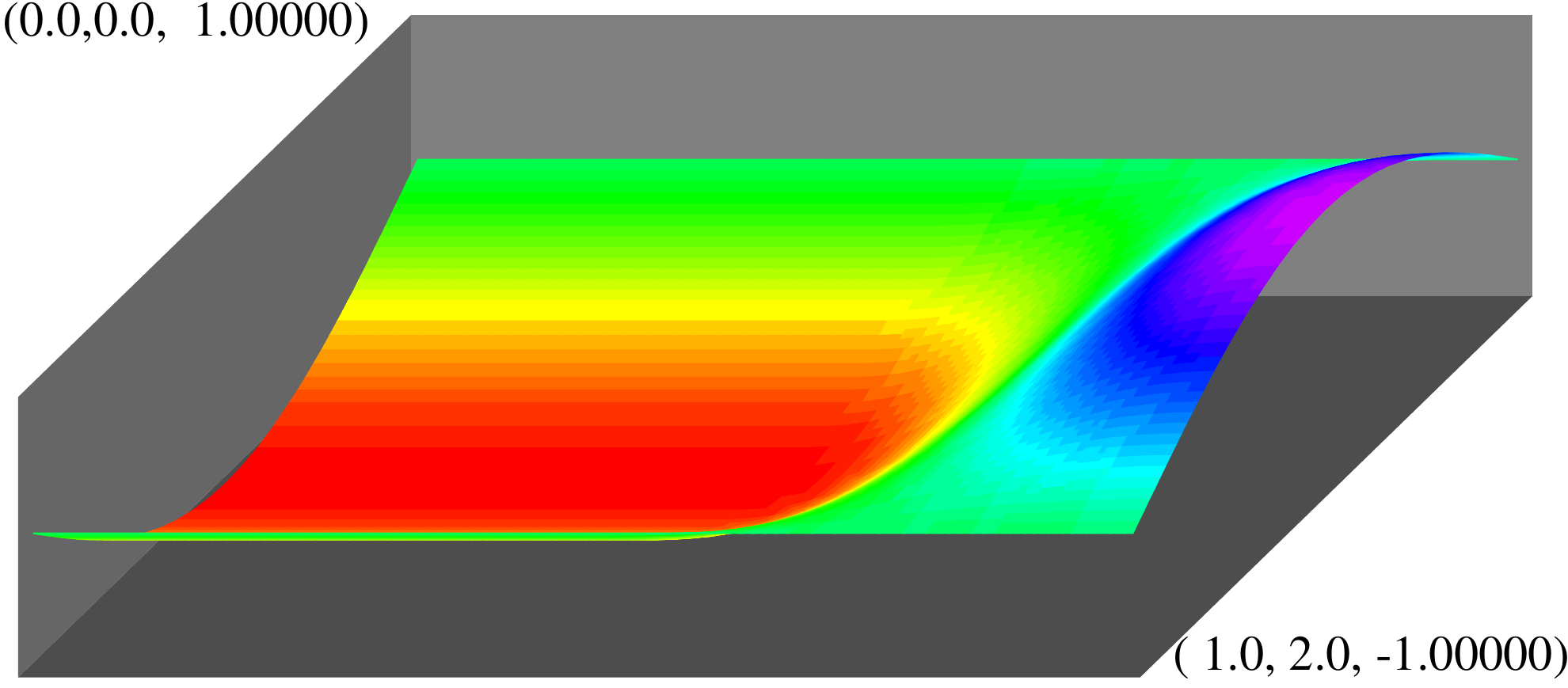}}
 \put(0,0){\includegraphics[width=2.2in]{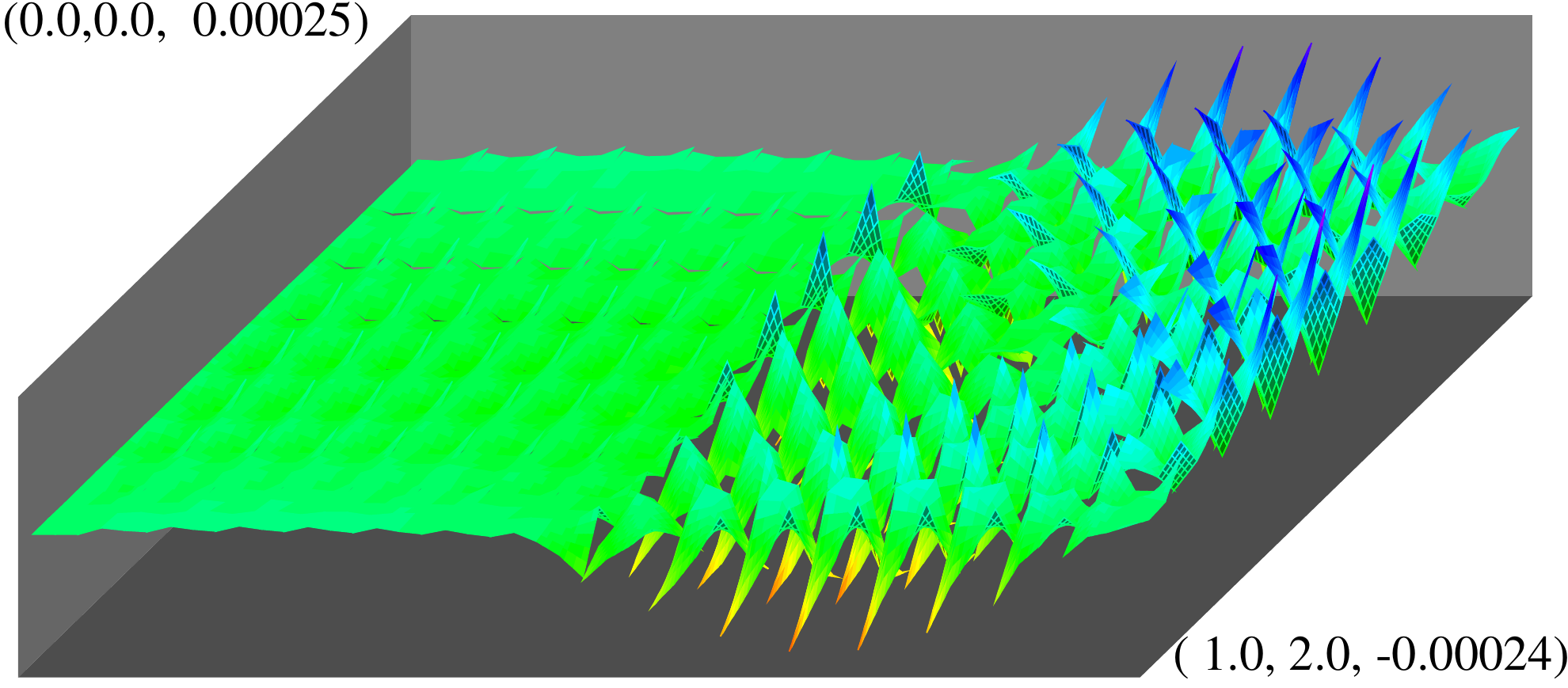}}
    \end{picture}
\caption{\label{g-u2} The solution for $(\bm u_s,\bm u_d)_2$ and its error by $P_3$ elements on level 4, for
    \eqref{s2}.  }
\end{center}
\end{figure}

  \begin{figure}[htb]\begin{center}\setlength\unitlength{1in}
   \begin{picture}(2.0,2.1)
 \put(0,1.2){\includegraphics[width=2.2in]{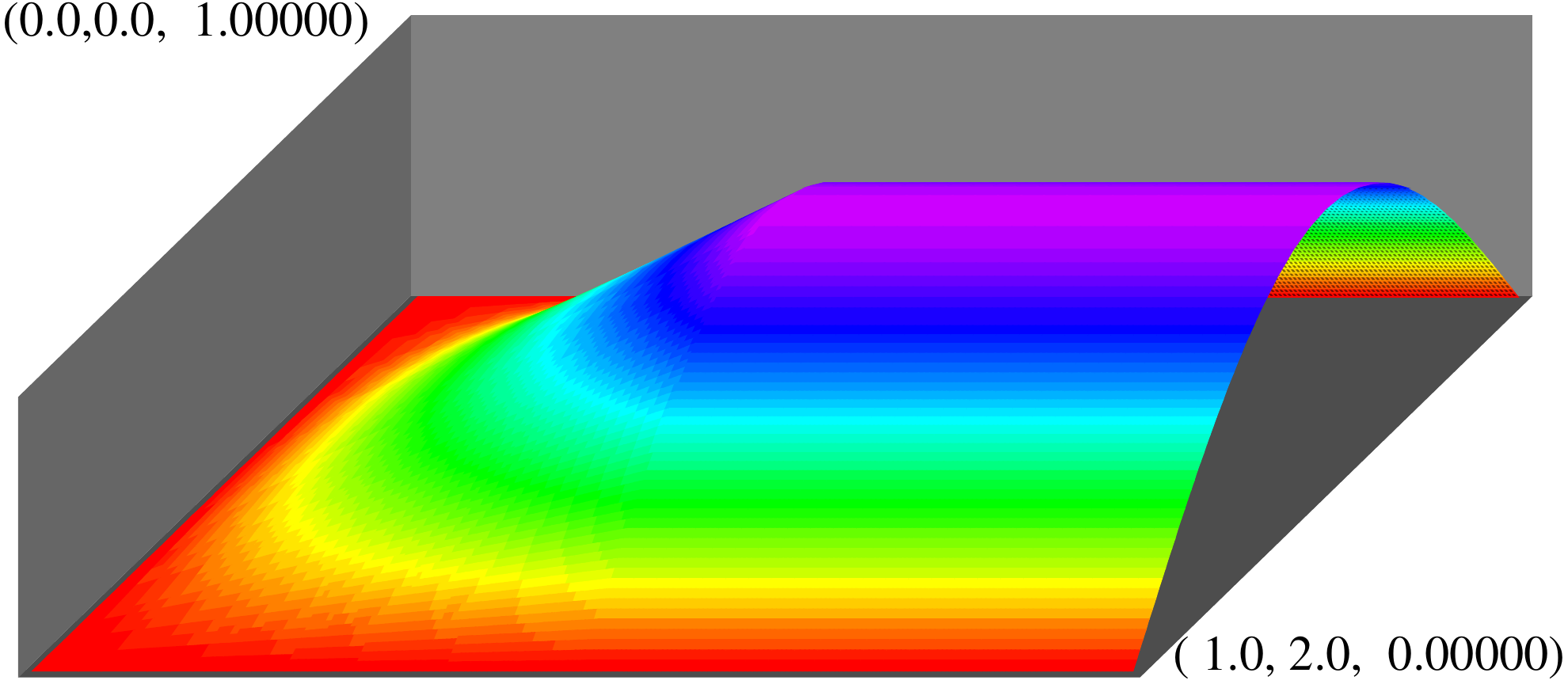}}
 \put(0,0){\includegraphics[width=2.2in]{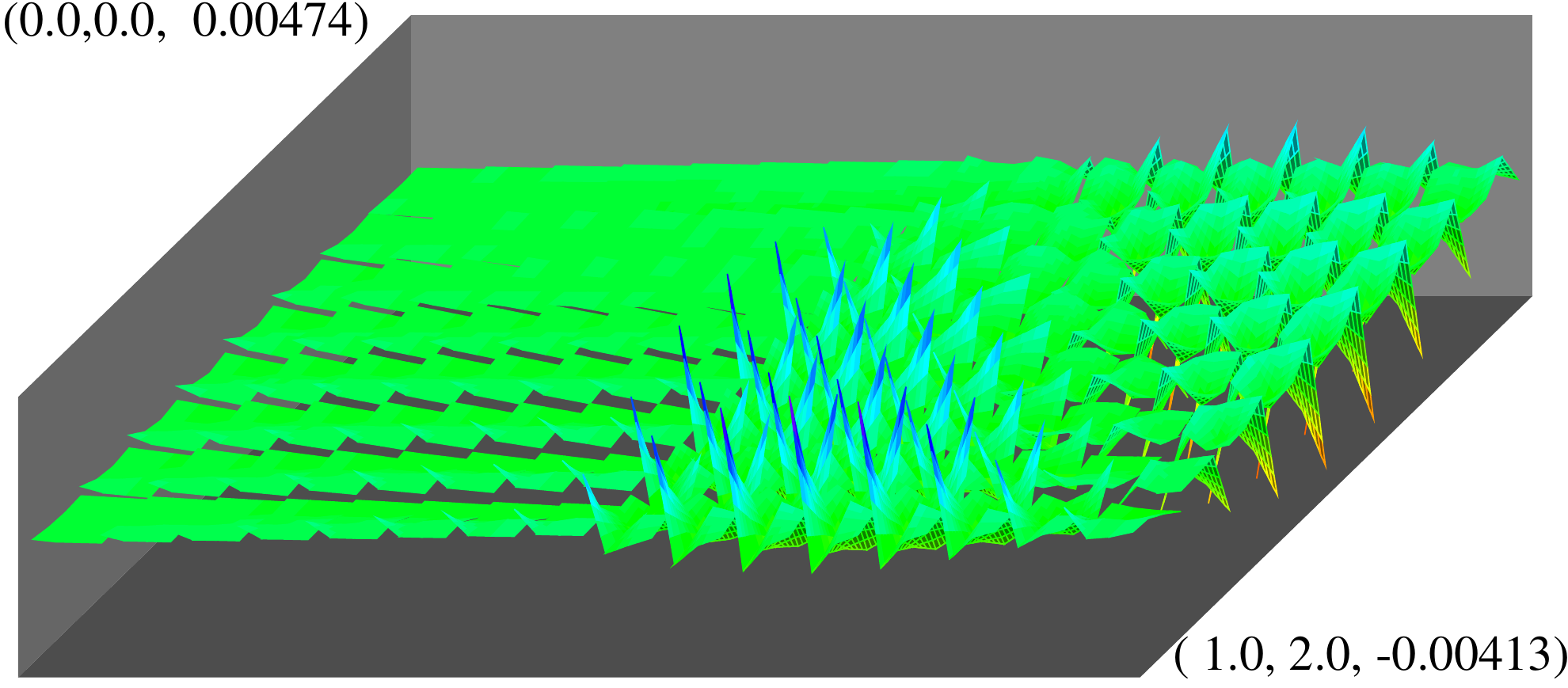}}
    \end{picture}
\caption{\label{g-p} The solution for $(p_s, p_d)$ and its error by $P_3$ elements on level 4, for
    \eqref{s2}.  }
\end{center}
\end{figure}

\begin{table}[htb]
  \caption{ The errors  and the order $O(h^k)$ of convergence by the
   $P_4$ WG and $BDM_4$ coupled element,  for \eqref{s2}.
   \label{t24} }
\begin{center}  \begin{tabular}{c|rr|rr|rr}  
\hline
level & $ \|Q_0 \bm u_s- \bm u_{s,h}\|_0 $ & $k$ &  $ \trb Q_h \bm u_s- \bm u_{s,h} \trb $
    & $k$ & $\|R^s_h p_s-p_{s,h}\|_0 $ & $k$  \\ \hline
 3& 0.1025E-03&  4.7& 0.4604E-02&  3.9& 0.5025E-03&  4.0 \\
 4& 0.4533E-05&  4.5& 0.2951E-03&  4.0& 0.3184E-04&  4.0 \\
 5& 0.2428E-06&  4.2& 0.1862E-04&  4.0& 0.1997E-05&  4.0 \\
\hline
   & $ \|I_h \bm u_d-\bm u_{d,h}\|_0 $ & $k$ &
   \multicolumn{2}{r|}{ $ \|\operatorname{\mbox{div}} (\bm u_d-\bm u_{d,h})\|_0\; k $} &
    $ \|I_h p_d-p_{d,h}\|_0 $ &$k$  \\ \hline
 3& 0.8851E-04&  4.0& 0.7961E-03&  4.0& 0.6335E-04&  4.0 \\
 4& 0.5590E-05&  4.0& 0.5017E-04&  4.0& 0.3988E-05&  4.0 \\
 5& 0.3506E-06&  4.0& 0.3142E-05&  4.0& 0.2497E-06&  4.0 \\
     \hline
\end{tabular}\end{center}
\end{table}

To see if we have different $L^2$-convergence, we compute the second example by the coupled $P_k$ WG vector and $P_k$ CG scalar elements with $k$ varying from 1 to 5. The corresponding results are recorded in Table \ref{t25}. The observed $L^2$ convergence orders of the velocity are $k+1$ for all polynomial degrees, as predicated by the theory of the $P_k$ WG elements for the Stokes equations.  In particular,  we have another order higher $L^2$-convergence for the $P_2$ element in the Darcy region. It behaves the same as in solving a pure Darcy problem.

To the best of our knowledge, there exists no general analysis for optimal error estimates of the velocity in $L^2$ norm. Fortunately, some researchers have noticed this problem and made efforts for some specific scheme (such as a monolithic strongly conservative numerical scheme) $\cite{L2-2, L2-1}$. But these work still cannot explain the
$L^2$ convergence in our study. We will explore the phenomenon in the future work.

\begin{table}[htb]
  \caption{ {\color{black}The errors  and the orders $O(h^k)$ of convergence,  for \eqref{s2}.
   \label{t25} }}
\begin{center}  \begin{tabular}{c|rr|rr|rr}  
\hline
level & $ \|Q_h \bm u_s- \bm u_{s,h}\|_0 $ & $k$
    & $\|Q_h p_s-p_{s,h}\|_0 $ & $k$
  & $ \|I_h p_d-p_{d,h}\|_0 $ & $k$ \\ \hline
  &\multicolumn{6}{c}{By coupled $P_1$ WG vector and $P_1$ CG scalar element. } \\ \hline
 4& 0.3625E-01&  1.8& 0.1223E+00&  0.9& 0.8650E-02&  1.9\\
 5& 0.9363E-02&  2.0& 0.6141E-01&  1.0& 0.2181E-02&  2.0\\
 6& 0.2361E-02&  2.0& 0.3059E-01&  1.0& 0.5441E-03&  2.0\\
\hline
  &\multicolumn{6}{c}{By coupled $P_2$ WG vector and $P_2$ CG scalar element. } \\ \hline
 4& 0.1682E-02&  2.9& 0.1261E-01&  2.0& 0.4920E-04&  3.7 \\
 5& 0.2126E-03&  3.0& 0.3138E-02&  2.0& 0.3755E-05&  3.7 \\
 6& 0.2664E-04&  3.0& 0.7792E-03&  2.0& 0.2994E-06&  3.6 \\
\hline
  &\multicolumn{6}{c}{By coupled $P_3$ WG vector and $P_3$ CG scalar element. } \\ \hline
 4& 0.6905E-04&  4.0& 0.6543E-03&  3.0& 0.2390E-05&  4.1 \\
 5& 0.4374E-05&  4.0& 0.8262E-04&  3.0& 0.1435E-06&  4.1 \\
 6& 0.2748E-06&  4.0& 0.1037E-04&  3.0& 0.8798E-08&  4.0 \\
\hline
  &\multicolumn{6}{c}{By coupled $P_4$ WG vector and $P_4$ CG scalar element. } \\ \hline
 3& 0.8282E-04&  4.9& 0.4249E-03&  3.8& 0.3119E-05&  5.0\\
 4& 0.2650E-05&  5.0& 0.2744E-04&  4.0& 0.9716E-07&  5.0\\
 5& 0.8342E-07&  5.0& 0.1731E-05&  4.0& 0.3025E-08&  5.0\\
\hline
  &\multicolumn{6}{c}{By coupled $P_5$ WG vector and $P_5$ CG scalar element. } \\ \hline
 3& 0.5416E-05&  5.9& 0.2996E-04&  4.9& 0.1474E-06&  6.0 \\
 4& 0.8596E-07&  6.0& 0.9595E-06&  5.0& 0.2286E-08&  6.0 \\
 5& 0.1349E-08&  6.0& 0.3020E-07&  5.0& 0.3557E-10&  6.0 \\
     \hline
\end{tabular}\end{center}
\end{table}

\section{Conclusion}
{\color{black}In this paper, the weak Galerkin finite element method coupled with the mixed finite element method is introduced for the Stokes-Darcy problem. We designed the numerical scheme and derived the optimal error estimates in broken $H^1$ norm for velocity and in $L^2$ for pressure. We found that the convergence order of velocity in $L^2$ norm is not always optimal form the numerical experiments. This phenomenon is strange and it will be studied in the following work.}

\begin{appendix}
\section{Some Technique Tools}
In this Appendix, we are going to introduce some technical results which have been used in previous section to derive error estimates.
\begin{lemma}\label{interpolation} Let $\T_{s,h}$ be a finite element partition of
domain $\Omega_s$ satisfying the shape regularity assumptions as
specified in \cite{Wang2014a}, we assume $\bw$ and
$\rho$ are sufficiently smooth. Then, for $0\le m\le 1$
we have

\begin{eqnarray}\label{pro-est1}
&&\sumT h^{2m}_{T_s}\|\bw-Q_0\bw\|^2_{T_s}\le Ch^{2(r+1)}_{T_s}\|\bw\|^2_{T_s,r+1}, \qquad 1\leq r\leq \alpha_s,
\\
\label{pro-est4}&&\sumT h^{2m}_{T_s}\|\bw-Q_b\bw\|^2_{e_s}\le Ch^{2(r+1)}_{T_s}\|\bw\|^2_{e_s,r+1}, \qquad 1\leq r\leq \beta,
\\
\label{pro-est2}&&\sumT h^{2m}_{T_s}\|\nabla\bw-\bQ_h(\nabla\bw)\|^2_{T_s}\le Ch^{2r}_{T_s}\|\bw\|^2_{T_s,r+1}, \qquad 1\leq r\leq \beta,
\\
\label{pro-est3}&&\sumT h^{2m}_{T_s}\|\rho-\dQ_h\rho\|^2_{T_s}\le Ch^{2r}_{T_s}\|\rho\|^2_{T_s,r}, \qquad 1\leq r\leq \beta.
\end{eqnarray}
Here $C$ denotes a generic constant independent of the mesh size
$h$ and the functions in the estimates.
\end{lemma}

\begin{lemma}\label{MFE-pro}
$\Pi^d_h$ satisfies the approximation properties
\begin{eqnarray}
\|\bv_d-\Pi^d_h\bv_d\|_0,T&\leq& Ch^m_{T_d}|\bv_d|_{m,T_d},\qquad 1\leq~m\leq~\alpha_d+1,\\
\|\nabla\cdot(\bv_d-\Pi^d_h\bv_d)\|_0,T&\leq& Ch^m_{T_d}|\nabla\cdot\bv_d|_{m,T_d},\qquad 0\leq~m\leq~\gamma_d+1.
\end{eqnarray}
\end{lemma}

\begin{lemma}
Let $p|_{\Omega_s}\in H^{\gamma_s}(\Omega_s)$, $p|_{\Omega_d}\in H^{\gamma d}(\Omega_d)$, then we have
\begin{eqnarray}
\|p-R_hp\|_{m,T}&\leq& Ch^{\gamma_s-m}|p|_{\gamma_s,T}\qquad T\in \Omega_s, ~m=0,1,\\
\label{p-pro}\|p-R_hp\|_{m,T}&\leq& Ch^{\gamma_d-m}|p|_{\gamma_d,T}\qquad T\in \Omega_d, ~m=0,1.
\end{eqnarray}
\end{lemma}

Let $T_s$ be an element satisfying the assumption verified in \cite{Wang2014a}
with $e_s$ as a side. For any function $g\in H^1(T)$, the
following trace inequality has been proved in \cite{Wang2014a}
\begin{eqnarray}\label{trace-thm}
\|g\|_{e_s}^2\le C(h_{T_s}^{-1}\|g\|_{T_s}^2+h_{T_s}\|\nabla g\|_{T_s}^2).
\end{eqnarray}
Particularly, if $g$ is polynomial in $T$ we have the
inverse inequality \cite{Wang2014a}
\begin{eqnarray}\label{inverse-thm}
\|\nabla g\|_{T_s}^2\le C h_{T_s}^{-2}\|g\|_{T_s}^2,
\end{eqnarray}
where $C$ is a constant only related to the degree of polynomial
and the dimension.
Combining with the trace inequality we can get further that
\begin{eqnarray}\label{trinv-thm}
\|\nabla g\|_e^2\le C h_{T_s}^{-1}\|g\|_{T_s}^2.
\end{eqnarray}
The vector version of the trace theorem and the inverse
theorem are trivial.
\begin{lemma}
For any $\bv_{_{s,h}}\in V^s_h$, we have
\begin{eqnarray}\label{rem-est1}
\sumTa \|\bv_{_{s,0}}-\bv_{{s,b}}\|_{\partial T_{_s}}&\le&C h^{\frac12}_{s}\|\bv_{s,h}\|_{V^s_h}.
\end{eqnarray}
\end{lemma}
\begin{proof}
When $\alpha_s=\beta$, \eqref{rem-est1} is obvious.
So we only need to discuss the case that $\alpha_s=\beta+1$.
We only consider the vector valued function $\bv_{_{S,h}}$. From Lemma $\ref{lemma-korn}$ we have
\begin{eqnarray}\label{formula1}
\sumTa \|\nabla\bv_{_{s,0}}\|_{T_{s}}\le C\|\bv_{s,h}\|_{V^s_h}.
\end{eqnarray}
Using the trace inequality $\ref{trace-thm}$ and $Poincar\acute{e}$ inequality, we can obtain that
\begin{eqnarray*}
\sumTa\|\bv_{{s,0}}-\bv_{{s,b}}\|_{\partial T_s}
&\le&\sumTa\|\bv_{{s,0}}-Q_b\bv_{{s,0}}\|_{\partial T_s}
+\sumTa\|Q_b\bv_{{s,0}}-\bv_{{s,b}}\|_{\partial T_s}
\\
&\le& C\sumTa h_{{T_s}}\|\nabla\bv_{{s,0}}\|_{\partial T_s}
+\sumTa\|Q_b\bv_{{s,0}}-\bv_{{s,b}}\|_{\partial T_s}
\\
&\le& C h_{s}^{\frac12}\sumTa\|\nabla\bv_{_{s,0}}\|_{T_s}
+h_{_s}^{\frac12}\sumTa h_{{T_s}}^{-\frac12}\|Q_b\bv_{_{s,0}}-\bv_{_{s,b}}\|_{\partial T_s}
\\
&\le& C h_{s}^{\frac12}\|\bv_{{s,h}}\|_{V^s_h},
\end{eqnarray*}
which completes the proof.
\end{proof}

\begin{lemma}\label{equ-err}
Let $\bw|_{\Omega_s}\in [H^{\alpha_s}(\Omega_s)]^2$, $\rho|_{\Omega_s}\in H^{\gamma_s}(\Omega_s)$, $i=s,~d$, and $\bv\in V_{s,h}$. Assume that the finite element partition $\mathcal{T}_{s,h}$ is shape regular. Then we have the  following estimates
\begin{eqnarray}
\label{l_1}l_1(\bw_s,\bv_{s,h})&\leq&Ch_s^{\beta+1}\|\bw_s\|_{\beta+2,\Omega_s}\|\bv_{s,h}\|_{V^s_h},\\
\label{l_2}l_2(\rho_s,\bv_{s,h})&\leq&Ch_s^{\gamma_s+1}\|\rho_s\|_{\gamma_s+1}\|\bv_{s,h}\|_{V^s_h},\\
\label{l_3}l_3(\bw_s,\bv_{s,h})&\leq&Ch^{\beta+1}\|\bw_s\|_{\beta+1,\Gamma}\|\bv_{s,h}\|_{V^s_h},\\
\label{stab}s(Q_h\bw_s,\bv_{s,h})&\leq& Ch_s^{\alpha_s}\|\bw_s\|_{\alpha_s+1}\|\bv_{s,h}\|_{V^s_h}.
\end{eqnarray}
\end{lemma}
\begin{proof}
Using Cauchy Schwarz inequality, $(\ref{rem-est1})$ and $(\ref{pro-est2})$, we have
\begin{eqnarray*}
l_1(\bw_s,\bv_{s,h})&=&2\nu\sum_{T\in\mathcal{T}_{s,h}}\langle \bv_{s,0}-\bv_{s,b}, D(\bw_s)\cdot\bn-(\mathbf{Q}_hD(\bw_s))\cdot\bn\rangle_{\partial T_s}\\
&\leq&C\left( \sum_{T\in \mathcal{T}_{s,h}}h^{-1}_{T_s}\|\bv_{s,0}-\bv_{s,b}\|^2_{\partial T_s}\right)^{1/2}\left(\sum_{T\in\mathcal{T}_{s,h}}h_{T_s}\|D(\bw_s)-\mathbf{Q}D(\bw_s)\|^2_{\partial T_s}\right)^{1/2}\\
&\leq&Ch_s^{\beta+1}\|\bw_s\|_{\beta+2,\Omega_s}\|\bv_{s,h}\|_{V^s_h}.
\end{eqnarray*}

The similarly technique can be applied to the following estimate,
\begin{eqnarray*}
l_2(\rho_s,\bv_{s,h})&=&\sum_{T_s\in\mathcal{T}_{s,h}}\langle \bv_{s,0}-\bv_{s,b},(\rho_s-R^s_h\rho_s)\bn\rangle_{\partial T_s}\\
&\leq&\left( \sum_{T\in \mathcal{T}_{s,h}}h^{-1}_{T_s}\|\bv_{s,0}-\bv_{s,b}\|^2_{\partial T_s}\right)^{1/2}\left( \sum_{T\in\mathcal{T}_{s,h}}h_{T,s}\|\rho_s-R^s_h\rho_s\|^2_{\partial T_s}\right)^{1/2}\\
&\leq&Ch_s^{\gamma_s+1}\|\rho_s\|_{\gamma_s+1}\|\bv_{s,h}\|_{V^s_h}.
\end{eqnarray*}
By the definition of the norm $\|\cdot\|^s_V$ and $(\ref{pro-est4})$,
\begin{eqnarray*}
l_3(\bw_s,\bv_{s,h})&=&\sum_{e\in\Gamma_h}\langle\mu\mathbb{K}^{-\frac12}(\bu_{s}-Q_b\bw_s)\cdot\tau,\bv_{s,b}\tau\rangle_e\\
&\leq&C\|\bw_s-Q_b\bw_s\|_{\Gamma}\|\bv_{s,h}\|_{V^s_h}\\
&\leq&Ch_s^{\beta+1}\|\bw_s\|_{\beta+1,\Gamma}\|\bv_{s,h}\|_{V^s_h}.
\end{eqnarray*}
Finally, from the property of $Q_h$, trace inequality $(\ref{trace-thm})$, we know that
\begin{eqnarray*}
s(Q_h\bw_s,\bv_{s,h})&=&\sum_{T\in \mathcal{T}_{s,h}}h^{-1}_{T_s}\langle Q_0\bw_s-\bw_{s},Q_b\bv_{s,0}-\bv_{s,b}\rangle_{\partial T_s}\\
&\leq&C\left( \sum_{T_s\in\mathcal{T}_{s,h}}h^{-1}_{T_{s}}\|Q_0\bw_s-\bw_s\|^2_{\partial T_s}\right)^{1/2}\left(\sum_{T\in\mathcal{T}_{s,h}}h^{-1}_{T_s}\|Q_b\bv_{s,0}-\bv_{s,b}\|^2_{\partial T_s}\right)^{1/2}\\
&\leq& Ch_s^{\alpha_s}\|\bw_s\|_{\alpha_s+1}\|\bv_{s,h}\|_{V^s_h}.
\end{eqnarray*}

\end{proof}

\end{appendix}

\end{document}